\documentclass[11pt]{article}
\usepackage{amsfonts}
\usepackage{latexsym}   
\usepackage{amsmath,amssymb,amsthm,latexsym}

\input cyracc.def
\newcommand{\Ker}{\mathrm{Ker}}
\newcommand{\Rad}{\mathrm{Rad}}
\newcommand{\lex}{\,\overrightarrow{\times}\,}

\begin{document}
\newtheorem{theorem}{Theorem}[section] \newtheorem{lemma}[theorem]{Lemma}
\newtheorem{definition}[theorem]{Definition}
\newtheorem{example}[theorem]{Example}
\newtheorem{remark}[theorem]{Remark}
\newtheorem{corollary}[theorem]{Corollary}
\newtheorem{proposition}[theorem]{Proposition}
\newtheorem{problem}[theorem]{Problem}
\baselineskip 0.24in
\title{\Large\bf Discrete $(n+1)$-valued states and $n$-perfect  pseudo-effect algebras}
\author{{Anatolij Dvure\v{c}enskij$^{1,2}$, Yongjian Xie$^{1,3}$\thanks{E-mail: yjxie@snnu.edu.cn},\ \ and
 Aili  Yang$^{4}$}\\ {}\\
 {\small $^1$Mathematical Institute, Slovak Academy of Sciences, \v{S}tef\'{a}nikova 49, SK-814 73 Bratislava, Slovakia }\\
 {\small $^2$Depar.
Algebra  Geom.,  Palack\'{y} Univer.,
CZ-771 46 Olomouc, Czech Republic,  dvurecen@mat.savba.sk}\\
 {\small $^3$College of Mathematics and Information  Science, Shaanxi Normal University, Xi'an, 710062, China}\\
  {\small $^4$College of Science, Xi'an University of Science and Technology, Xi'an, 710054, China}
 }

\date{}
\maketitle
\begin{center}
\begin{minipage}{140mm}

\begin{abstract}
We give  sufficient and necessary conditions  to guarantee that a pseudo-effect algebra
admits an $(n+1)$-valued discrete state. We introduce $n$-perfect pseudo-effect algebras as algebras which can be split into $n+1$ comparable
slices. We prove that the category of strong $n$-perfect pseudo-effect algebras is categorically equivalent to the category
of torsion-free directed partially ordered groups of a special type.
\end{abstract}

 \vskip 2mm \noindent{\bf Keywords}:  Pseudo-effect algebra; po-group; symmetric pseudo-effect algebra; $(n+1)$-valued discrete state;
$n$-perfect pseudo-effect algebra; lexicographic product;
categorical equivalence

 \vskip 2mm \noindent{\bf MSC2000}: 03G12; 03B50; 08A55; 06B10

\end{minipage}
\end{center}

\section  {Introduction and basic definitions}

Effect algebras were introduced by Foulis and Bennett in 1994 for modeling unsharp measurements of a quantum mechanical system
\cite{FoBe94}. They are a generalization of many structures which arise in quantum mechanics, in particular of orthomodular lattices
in noncommutative measure theory and of MV-algebras in fuzzy measure theory \cite{DvPu00, Kalmach83, Hajek03}. Alternative structures called
difference posets, which are categorically equivalent with effect algebras, were introduced in \cite{KoCh94}. At the end of the 90's,
a
noncommutative version of MV-algebras, called pseudo-MV-algebras,  was introduced in \cite{GeIo01} and independently in \cite{Rac02} as generalized MV-algebras, GMV-algebras.  A noncommutative generalization of effect algebras, called pseudo-effect algebras, PEAs, was introduced and studied in \cite{DvVe01a,DvVe01b}.

Perfect MV-algebras, introduced by Belluce, Di Nola and Lettieri in \cite{BDL93}, may be viewed as the most compelling examples of
non-Archimedean MV-algebras, in such a sense that  they are generated by their infinitesimals. In a perfect MV-algebra, every
element belongs either to its radical or its coradical. In \cite{DL94}, it is shown that for any  perfect MV-algebra $M$, there exists an Abelian unital $\ell$-group $G$ (lattice-ordered group), such that $M$ is isomorphic to an interval of the lexicographic product
of the group of integers $\mathbb{Z}$ and $G.$

Especially, MV-algebras are lattice-ordered effect algebras satisfying the Riesz Decomposition Property (RDP). For any effect algebra $E$ with (RDP), there exists a partially ordered group $G$ with a strong unit $u$  such that $E$ is isomorphic to the effect algebra $\Gamma(G,u):=[0,u],$ where the effect algebra operations are the group additions existing in $[0,u].$ In \cite{Dv07},  Dvure\v{c}enskij introduced  perfect effect algebras, which are one kind of effect algebras admitting (RDP), and he proved that every perfect algebra is an interval in the lexicographical product of the group of integers with an Abelian directed  partially ordered group with interpolation. Moreover,
Dvure\v{c}enskij showed that the category of perfect effect algebras is categorically equivalent to the category of Abelian directed
partially ordered groups with interpolation.

The principal result on a representation of GMV-algebras says that every GMV-algebra is always  an interval in a
unital $\ell$-group \cite{Dv02}, i.e. an $\ell$-group with strong unit. This result provides a new bridge between different research areas, including GMV-algebras, unital
$\ell$-groups, noncommutative many valued logic, soft computing and quantum structures \cite{DvPu00}. Especially, using this result, perfect GMV-algebras and $n$-perfect GMV-algebras were introduced in \cite{DDT,Dv08}.  Furthermore, the author proved that any  $n$-strong perfect GMV-algebra is always an interval in the lexicographical product of the group of integers  $\mathbb{Z}$ with an $\ell$-group.

The notion of a state, as an analogue of a probability measure, is crucial for quantum mechanical measurements. Therefore, a special attention is done in order to exhibit whether does a state exist for the studied structure and if yes, what are its basic properties. In any Boolean algebra, we have a lot of two-valued states, and in general, every two valued state is extremal. Every perfect MV-algebra or every perfect effect algebra admits only a two-valued state. On the other hand, in the orthodox example of quantum mechanics, see e.g. \cite{93}, the system $L(H)$ of all closed subspaces of a Hilbert space $H$, $\dim H\ge 3$, does not admit any two-valued state. A two-valued state is a special case of an $(n+1)$-valued discrete state, and in this paper we concentrate to the existence of such states. We recall that, if $\dim H=n,$ then $L(H)$ admits a unique  $(n+1)$-valued  discrete state $s$, namely $s(M)= \dim M/n,$  $M \in L(H).$  In general, discrete states are of particular importance in many areas of mathematics, see \cite{Good86}.  For example, discrete states were used in \cite[Thm 3.2]{209} in order to completely  describe monotone $\sigma$-complete effect algebras.

GMV-algebras are also PEAs,  and so, $n$-perfect GMV-algebras are PEAs. Any $n$-perfect GMV-algebra  admits an $(n+1)$-valued discrete state. The main accent of the present paper is done to the study of $(n+1)$-valued discrete states on PEAs and to show how they are related with $n$-perfect PEAs.

Therefore, we start with the study of the structure of PEAs admitting a two-valued state. Then we present one characterization of PEAs admitting an $(n+1)$-valued discrete state. Later, we give a definition of $n$-strong perfect PEAs, and we prove that any $n$-strong perfect GMV-algebra is always an interval in the lexicographical product of the group of integers   $\mathbb{Z}$ with a torsion-free partially ordered group $G$ such that $\mathbb Z\lex G$ satisfies a special type of the Riesz Decomposition Property. Finally, the relationships between the category of $n$-strong perfect PEAs and the category of torsion-free directed partially ordered groups of a special type are discussed.

The paper is organized as follows. Some basic definitions and properties about pseudo-effect algebras  are presented in Section 2. In
Section 3, we study the structure of pseudo-effect algebras with two-valued states. In Section 4, we give  sufficient and
necessary conditions to guarantee that a pseudo-effect algebra
admits an $(n+1)$-valued discrete state. In Section 5, we introduce the class of $n$-perfect pseudo-effect algebras.  In Section 6, we study the class of strong $n$-perfect pseudo-effect algebras and we prove that the category of strong $n$-perfect pseudo-effect algebras is
categorically equivalent to the category of torsion-free directed  partially ordered groups of a special type.

\section  {Basic definitions and facts}

In this section,  we give  basic definitions  and facts about pseudo-effect algebras which we will need in this paper.

\begin{definition}\label{def:GPEA} {\rm \cite{DvVe02} A structure  $(E;+,0),$  where $+$ is a partial binary
operation and 0 is a  constant, is called a {\it generalized pseudo-effect algebra}, or GPEA for short, if for all $a,b,c \in E,$ the following hold.
\begin{itemize}
\item[{\rm (GP1)}] $ a+ b$ and $(a+  b)+ c $ exist if and only if
$ b+ c$ and $a+ ( b+ c) $ exist, and in this case,
$(a+  b)+ c =a+ ( b+ c)$.

\item[{\rm (GP2)}] If $ a+ b$ exists, there are elements $d,e\in E$ such
that $a+ b=d+ a=b+ e$.

\item[{\rm (GP3)}] If $ a+ b$ and $a+ c $ exist and are equal, then
$b=c.$ If $b+ a$ and $c+ a $ exist and are equal, then $b=c.$

\item[{\rm (GP4)}] If $ a+ b$  exists and $a+ b=0$, then $a=b=0.$

\item[{\rm (GP5)}] $ a+ 0$ and $0+ a $ exist and both are equal to $a.$
\end{itemize}
}
\end{definition}

According to \cite{DvVe02},  we introduce a binary relation $\leqslant$ in a GPEA $E$. For $a,b\in E,$ we define $a\leqslant b$ if
and only if there is an element $c\in E$ such that $a+ c=b.$ Equivalently, there exists an element $d\in E$ such that $d+ a=b.$ Then $\leqslant$ is a partial order on $E$.

We introduce two partial binary operations $/$ and $\backslash$ on a GPEA $E$. For any $a,b\in E$, $a/ b$ is defined if and only if
$b\backslash a$ is defined if and only if $a\leqslant b$, and in such a case we have $(b\backslash a)+ a=a+ (a/ b)=b.$  Then $a=(b\backslash a)/b=b\backslash(a/ b).$

By  \cite{XieLi11}, a GPEA $E$ is called a {\it weakly commutative} GPEA if it satisfies the following condition:

(C)  for any $ a,b\in E, a+  b  $ exists  if and only if $b+ a$ exists .

\begin{definition}\label{def:pea}
{\rm  \cite{DvVe01a} A structure
$(E;+,0,1),$  where $+$ is a partial binary operation and 0 and 1 are constants, is called a {\it pseudo-effect algebra}, or PEA for short, if for all $a,b,c \in E,$ the following hold.
\begin{itemize}
\item[{\rm (PE1)}] $ a+ b$ and $(a+ b)+ c $ exist if and only if $b+ c$ and $a+( b+ c) $ exist, and in this case,
$(a+ b)+ c =a +( b+ c)$.

\item[{\rm (PE2)}] There are exactly one  $d\in E $ and exactly one $e\in E$ such
that $a+ d=e + a=1 $.

\item[{\rm (PE3)}] If $ a+ b$ exists, there are elements $d, e\in E$ such that
$a+ b=d+ a=b+ e$.

\item[{\rm (PE4)}] If $ a+ 1$ or $ 1+ a$ exists,  then $a=0$ .
\end{itemize}}
\end{definition}

Let  $a$ be an element of a PEA $E$ and $n\geqslant0$ be an integer.
We define $na=0$ if $n=0,$ $1a=a$ if $n=1,$ and $na=(n-1)a+ a$ if $(n-1)a$ and $(n-1)a+ a$ are defined in $E$.  We define the {\it isotropic index} $\imath(a)$ of the element $a$, as the maximal nonnegative number $n$ such that $na$ exists. If $na$ exists for
every  integer $n$, we say that $\imath(a)=+\infty.$ In the following, we denote by $\mbox{\rm Infinit}(E)=\{a\in E\mid \imath(a)=+\infty\}.$

We recall that if $(E;+,0,1)$ is a PEA, then $(E;+,0)$ is a GPEA. If $a+ b$ exists and $a+ b=1,$ then we write
$b^{-}=a,$ $a^{\sim}=b.$ Thus,  two mappings $a\mapsto a^{-}$ and $a\mapsto a^{\sim}$ are  unary operations  satisfying the following:

(i)\ if $ a\leqslant b,$ then $b^{-} \leqslant a^{-},$ $b^{\sim} \leqslant a^{\sim}$.

(ii)\ for any $a\in E,$ $a^{-}$$^\sim=a^{\sim}$$^{-}=a.$

Assume that $(G;+,\leqslant,0)$ is a  po-group (po-group for short), i.e. $G$ is a group written additively with a partial ordering $\leqslant$ such that $a\leqslant b$ implies $c+a+d\leqslant c+b+d$ for any $c,d \in G.$ A positive element $u \in G$ is said to be a {\it strong unit} if, given $g \in G$, there is an integer $n\ge 1$ such that $g\leqslant nu.$ The pair $(G,u)$ is said to be a {\it unital po-group}. A po-group is said to be {\it directed} if, $a,b \in G,$ there is an element $c\in G$ such that $a,b \leqslant c.$ For more about po-groups, see \cite{Gla}.

We denote by $G^+:=\{g\in G\mid g\geqslant 0\}.$ For any $x,y\in G^{+}$, let $x+y$ be the group addition of $x$ and $y.$ Then $(G^{+};+,0)$ is a generalized pseudo-effect algebra. We set
$\Gamma(G,u):=\{x\in G\mid 0\leqslant x\leqslant u\},$ and we endow $\Gamma(G,u)$ with the operation $+$ such that $a+ b$ is defined in $\Gamma(G,u)$ whenever  $a\leqslant u-b,$  and in such a case, $a+b$ in $\Gamma(G,u)$ is the group addition of $a$ and $b.$ Then $(\Gamma(G,u),+,0,u)$ is a pseudo-effect algebra. We say that a
PEA $E$ is an {\it interval} PEA if there exists a  po-group $G$ such that $E$ is isomorphic to a PEA $\Gamma(G,u)$ for some
strong unit $u\in G^{+}.$

We say that a PEA $E$ satisfies the {\it Riesz Decomposition Property} (RDP), if $a_1+a_2 =b_1+b_2$, there are four elements $c_{11}, c_{12}, c_{21}, c_{22}$ such that $a_1 = c_{11}+c_{12},$ $a_2 = c_{21}+c_{22},$ $b_1 = c_{11}+c_{21},$ and $b_2 = c_{21}+c_{22}.$  If, in addition, for all $x\leqslant c_{12}$ and $y\leqslant c_{21},$ we have $x+y$, $y+x$ exists in $E$ and $x+y=y+x,$ we say that $E$ satisfies (RDP)$_1$. By  \cite{DvVe01b}, every PEA $E$ with (RDP)$_1$ is an interval PEA.

A PEA $E$ satisfies  (RDP)$_0$  if, for any $a, b_1, b_2
\in E$ such that $a \leqslant b_1+b_2,$ there are $d_1, d_2 \in E$ such
that $d_1 \leqslant b_1$, $d_2 \leqslant b_2$ and $a = d_1+d_2$.

We recall that (RDP)$_1 \Rightarrow $ (RDP) $\Rightarrow$ (RDP)$_0$, but the converse is not true, see \cite{DvVe01b}.

Finally, we say  a directed po-group $G$  satisfies (RDP) or (RDP)$_1$ if the same property as for PEA holds also for the positive cone $G^+.$

We recall that according to \cite[Thm 5.7]{DvVe01b}, for any PEA $E$ with (RDP)$_1$ there is a unique (up to isomorphism of unital po-groups) unital po-group $(G,u),$ where $G$ satisfies (RDP)$_1$, such that $E\cong \Gamma(G,u).$ Moreover, there is a categorical equivalence between the category of PEAs with (RDP)$_1$ and the category of unital po-groups $(G,u),$ where $G^+$ satisfies (RDP)$_1$.

\begin{example}\label{ex:ex21}
{\rm Let $\mathbb{Z}$ be the group of integers and $G=\mathbb{Z}$$\times\mathbb{Z}$$\times\mathbb{Z}.$ Define for every two elements of $G$

\begin{displaymath}
(a,b,c)+ (x,y,z)=\left\{
\begin{array}{ll}
(a+x,b+y,c+z), & \textrm{$x$ is even,}\\
(a+x,c+y,b+z), & \textrm{$x$ is odd,}\qquad\qquad
\end{array} \right.
\end{displaymath}
and define $(a,b,c)\leqslant(x,y,z)$  if $a<x$ or
$a=x,b\leqslant y$ and $c\leqslant z$.

Then $(G;+,\leqslant)$ is a lattice-ordered group with  strong unit $u=(1,0,0),$  and $\Gamma(G,(1,0,0))$ is a PEA.}
\end{example}

\begin{proposition}\label{pr:complement}
{ Let  $E$ be a  PEA. Then $E$ is a weakly
commutative if and only if,  for any $a\in E,$ $a^{-}=a^{\sim}.$ }
\end{proposition}

\begin{proof}  Assume that, for any $a\in E$, we have that $a^{-}=a^{\sim}.$ If $a+ b$ exists in $E$ for $a,b\in E$, then $
b\leqslant a^{\sim}=a^{-},$  which implies that $b+ a$ exists in $E.$ Hence, $E$ is a weakly commutative PEA.

Conversely, assume that $E$ is a weakly commutative PEA. Then for any $a\in E$, $a^{-}+ a$ exists and equals  the unit $1$. By the condition (C), we have that $a+ a^{-}$ exists in $E$, and so $a^{-}\leqslant a^{\sim}$. Since the equality $a+ a^{\sim}=1$ holds, we have that $a^{\sim}+ a$ exists, and so $a^{\sim}\leqslant a^{-}$.  Hence, we conclude   $a^{\sim}= a^{-}.$
\end{proof}

\begin{remark}\label{def:t-norm}
{\rm  In \cite{Dv04}, Dvure\v{c}enskij has introduced  symmetric pseudo-effect algebras as following: a pseudo-algebra $E$ is said to
be {\it symmetric} (or, more precisely, with symmetric differences) if $a^{-} = a^{\sim} $ for any $a \in E$. Now, by Proposition
\ref{pr:complement}, the set of weakly commutative pseudo-effect algebras coincides  with the set of symmetric pseudo-effect algebras.  In such a case, we set $a^{\prime}=a^{-} = a^{\sim}$ and $a^{\prime}$ is said to be
an orthosupplement of $a.$
}
\end{remark}

\begin{example}\label{ex:ex22}
{\rm Let $\mathbb{Z}$ be the group of integers and $G$ be a (not necessarily Abelian) po-group. Let $\mathbb{Z}\overrightarrow{\times}G$ be the lexicographic product of $\mathbb{Z}$ and $G,$ then $\mathbb{Z}\overrightarrow{\times}G$ is a
directed po-group with strong unit $(1,0)$. If we set $E=\Gamma(\mathbb{Z}\overrightarrow{\times}G,(1,0)),$
then $E$ is a symmetric PEA but not necessarily commutative PEA. }
\end{example}

Let $(G;+,\leqslant,0)$ be a    po-group. An element $c\in G$  such that $x+c=c+x$ for all $x\in G$ is said to be a commutator of $G$. We set $C(G)=\{c\in G\mid c $ is a commutator of $G\}.$

\begin{example}\label{ex:symPEA}
{\rm Let $G$ be a po-group. Assume that $c\in G^{+}$ is a commutator, then
$E=\Gamma(\mathbb{Z}\overrightarrow{\times}G,(n,c))$   is a symmetric PEA. Especially,
$\Gamma(\mathbb{Z}\overrightarrow{\times}G,(n,0))$ is a symmetric PEA. }
\end{example}

Let $(E;$ $+,$ $0)$  be a  GPEA. Let $E^{\sharp}$ be a set disjoint from $E$ with the same cardinality. Consider a bijection
$a\mapsto a^{\sharp}$ from $E$ onto $E^{\sharp}$ and let us denote $E\cup E^{\sharp}=\hat{E}$. Define a partial operation $+^{*}$
on $\hat{E}$ by the following rules. For $a,b\in E,$

\begin{itemize}
\item[(i)] $a+^{*} b$ is defined if and only if $a+ b$ is defined, and $a+^{*} b:=a+ b.$
\vspace{-2mm}
\item[(ii)] $a+^{*} b^{\sharp}$ is defined if and only if $b\backslash a$ is
defined, and then $a+^{*} b^{\sharp}:=(b\backslash a)^{\sharp}.$

\vspace{-2mm}
\item[(iii)] $b^{\sharp}+ ^{*}a$ is defined if and only if $a/b$ is defined, and then $b^{\sharp}+^{*} a:=(a/b)^{\sharp}.$
\end{itemize}

In \cite{XieLi11}, we have obtained the following results.

\begin{proposition}\label{pr:symmetric}
If $(E;+,0)$ is  a symmetric GPEA, then the structure $(\hat{E};+^{*}, 0,0^{\sharp})$ is a symmetric PEA. Moreover,
$E$ is an order ideal in $ \hat{E}$ closed under $+,$ and the partial order induced by $+^{*} ,$ when restricted to $E,$
coincides with the partial order induced by $+$.
\end{proposition}

\begin{proposition}\label{pr:unitization}
Let $(E;+,0)$ be a GPEA and let the structure $(\hat{E};+^{*},0,0^{\sharp})$ be a PEA, then $(E;+,0)$ is a symmetric  GPEA and $(\hat{E};+^{*},0,0^{\sharp})$ is a symmetric PEA.
\end{proposition}

By Proposition  \ref{pr:symmetric} and \ref{pr:unitization}, we immediately conclude the following result.

\begin{corollary}\label{def:symgpea}
{Let  $(E;+,0)$ be a GPEA. Then the algebraic system $(\hat{E};+^{*},0,0^{\sharp})$ is a PEA if and only if $(E;+,0)$ is  a symmetric GPEA. }
\end{corollary}

The symmetric PEA $(\hat{E};+^{*},0,0^{\sharp})$ is usually
called the {\it unitization} of a symmetric GPEA $(E;+,0),$ and for any $a\in E,$ $a+ a^{\sharp}=a^{\sharp}+ a=0^{\sharp},$ hence, $a^{\prime}=a^{\sharp}$ and
$a^{\sharp}$$^{\prime}=a.$ Since the operation $+^{*}$ on $\hat{E}$ coincides with the $+$ operation on $E$, it will cause no confusion if we use the notation $+$ also for its extension on $\hat{E}.$

\begin{definition}\label{def:morphism}
{\rm (i)\ Let $E,F$ be two GPE-algebras. A mapping $f : E\rightarrow F$ is a {\it morphism} if the following conditions are satisfied:

(1)\ $f(0_{E}) = 0_{F}$ .

(2)\ If $a, b \in E$ and $a+ b$ exists, then $f(a)+ f(b)$
exists and $f(a+ b) = f(a)+ f(b)$.

(ii)\ Let $E,F$ be two pseudo-effect algebras. A mapping $f : E\rightarrow F$ is a {\it morphism} if the following conditions
are satisfied:

(1)\ $f(0_{E}) = 0_{F}, f(1_{E}) = 1_{F}$ .

(2)\ If $a, b \in E$ and $a+ b$ exists, then $f(a)+ f(b)$ exists and $f(a + b) = f(a)+ f (b)$.

(iii)\ Let $E$ be a PEA. Then any morphism
$s : E\rightarrow[0,1]$ is said to be a {\it state} on $E$. A state $s$ is said to be {\it discrete} if there exists an integer $n$ such that $s(E) \subseteq \{0,\frac{1}{n},\ldots,1\},$ where $s(E)=\{s(x)\mid x\in E\}.$ If $s(E)=\{0,\frac{1}{n},\ldots,1\}$, we say that $s$ is an $(n+1)$-valued discrete state.

Especially, if $n=1$, then we say that $s$ is a {\it two-valued  state}.

(iv)\ A state $s$ is said to be  {\it extremal} if, for any
states $s_{1},s_{2}$ and $\alpha\in(0,1)$, the equation $s=\alpha s_{1}+(1-\alpha)s_{2}$ implies $s=s_{1}=s_{2}.$ }
\end{definition}

Of course, every two-valued state is a 2-valued discrete state, and vice-versa.
For example, every two-valued state on a PEA $E$ is extremal.

We note that if $s$ is a state on $E$, then $s(0)=0$ and $s(1)=1,$ therefore, in what follows we will assume $0\ne 1.$

\begin{example}\label{ex:nsta}
{ Let $G$ be a directed po-group and let $c\in G.$  Then $\Gamma(\mathbb{Z}\overrightarrow{\times}G,(n,c))$  admits an $(n+1)$-valued discrete state.}
\end{example}

We recall that the real interval $[0,1]$ can be assumed also as an interval effect algebra $\Gamma(\mathbb R,1).$

\begin{remark}\label{re:2valued}
{\rm Let $E$ be a PEA and $s:E\rightarrow [0,1]$ be a   state with $|s(E)|=n+1.$

(i) If $n=1$, then $s(E)=\{0,1\}$ is a sub-effect algebra of $[0,1]$ and $s$
is a two-valued  state.

(ii) If $n> 1,$ then $s(E)$ is not necessarily a sub-effect algebra of $[0,1]$. For example, let $E=\{0,a,b,1\}$, we endow $E$ with the partial
operation $+$ as follows,
(1) for any $x\in E$, $x+0=0+ x=x,$ (2) $a+ b=b+ a=1$. Then the algebraic
system $(E;+,0,1)$ is an effect algebra. We define a mapping $s:E\rightarrow [0,1]$ as follows, $s(0)=0,$ $s(a)=\frac{2}{5},$
$s(b)=\frac{3}{5},$ $s(1)=1$, then $s$ is a discrete state on $E$ and $s(E)=\{0,\frac{2}{5},\frac{3}{5},1\}.$
However, $s(E)$ is not a sub-effect algebra of $[0,1].$

(iii) Let $s$ be an $(n+1)$- valued discrete state on an effect algebra. Then $s$ is not necessarily extremal. For example, for the effect algebra $E$ in (ii), set $s(0)=0,$ $s(a)=s(b)=\frac{1}{2},$
$s(1)=1,$ then $s$ is a 3-valued discrete state which is not extremal. In fact, we set $s_{1}(0)=s_{1}(a)=0,$
$s_{1}(b)=s_{1}(1)=1,$ and $s_{2}(0)=s_{2}(b)=0,$
$s_{2}(a)=s_{2}(1)=1,$ then $s_{1}, s_{1}$ are two states on $E,$ and $s=\frac{1}{2}s_{1}+\frac{1}{2}s_{2},$ however, $s\neq s_{1},$ $s\neq s_{2}.$  In  \cite[Prop 8.5]{Dv05}, Dvure\v{c}enskij has proved that if $s$ is an extremal discrete state on an effect algebra $E$ with (RDP), then $s$ is an $(n+1)$-valued discrete state. }
\end{remark}

\begin{theorem}\label{th:mvalued}
Let $E$ be a PEA and $s:E\rightarrow [0,1]$ be a state. Assume that $|s(E)|=n+1$ and  $n\geqslant1$. Then the following statements are equivalent.

\begin{itemize}
\item[{\rm (i)}] $s$ is an $(n+1)$-valued discrete state.

\vspace{-2mm}
\item[{\rm (ii)}] $s(E)$ is a sub-effect algebra of the effect algebra $[0,1].$

\vspace{-2mm}\item[{\rm (iii)}] For any $t,u\in s(E)$, if $t\leqslant u$, then there exists a $v\in s(E)$, such that $t+v=u.$
\end{itemize}
\end{theorem}

\begin{proof}
If $n=1,$ then using Remark \ref{re:2valued}, it is easy to see that
(i), (ii) and (iii) are mutually equivalent.

Now, we assume that $n>1$ and $s(E)=\{0,t_{1},\ldots,t_{n-1},1\}$,
where $0<t_{1}<t_{2}<\cdots <t_{n-1}<1.$

(i)$\Rightarrow$(ii). If $s$ is an $(n+1)$-valued discrete state, then
$s(E)=\{0,\frac{1}{n},\ldots,\frac{n-1}{n},1\}$ is a sub-effect algebra of $[0,1].$

(ii) $\Rightarrow$ (iii). Assume that $s(E)$ is a sub-effect algebra of the effect algebra $[0,1].$ For any   $t,u\in s(E)$, if     $t\leqslant u$, then there exists a $v\in s(E)$ such that $t+ v$ exists and $t+ v=u.$

(iii) $\Rightarrow$ (ii).
 By (iii), for any $t,v\in s(E)$ with $t\leqslant v,$ we have that $v-t\in s(E)$. We define a partial binary operation $+$ on $s(E)$ as follows:  $t+ v$ exists in $s(E)$ iff $t\leqslant 1-v$, and then  $t+v$ is the classical addition of two real numbers $t$ and $v.$ It is routine to verify that $(s(E);+,0,1)$ is an effect algebra. Further, for any $t,v\in s(E),$ $t+ v$ exists iff $t+v\leqslant1,$ which implies that  $(s(E);+,0,1)$ is a sub-effect algebra of [0,1].

(ii) $\Rightarrow $ (i).
Assume that (ii) holds, and so (iii) holds, too. It suffices to prove
that $t_{i}=\frac{i}{n}$ for any $i\in\{1,\ldots,n-1\}$.

If $n=2$, then we have that $s(E)=\{0,t_{1},1\}$. By  $0<t_{1}<1,$
there exists a real number $t\in s(E)$ such that $t+ t_{1}=1$. Obviously,
$t\neq0,1$, and so $t=t_{1}$, which implies that
$t_{1}=\frac{1}{2}.$  Then (i) holds.

Now, we assume that  $n>2.$ We are claiming that for any $i\in \{1,\ldots, n-1\},$ $t_i=it_1.$

If $i=2,$ then  $t_{1}<t_{2},$ and there exists a $j\in \{1,2\}$ such that $t_{1}+ t_{j}=t_{2},$ and so $j=1$, hence
$t_{2}=2t_{1}$.

Assume by  induction that, for any $j\leqslant i< n-1,$ we have proved $t_j=jt_1.$ Since $t_i<t_{i+1},$ we have $t_{i+1}-t_i \in \{t_1,\ldots, t_i\}.$ If it would be $t_{i+1}-t_i \ge 2t_1,$ then $t_i <t_i+t_1 <t_i +2t_1 \leqslant t_i+t_j = t_{i+1}$ which is impossible because between $t_i$ and $t_{i+1}$ there is no element in $s(E).$ Hence, $t_{i+1}-t_i = t_1$ which proves $t_i=it_1$ for any $i=1,\ldots, n-1.$

Finally, $0<1-t_{n-1}<\cdots <1-t_1 <1$ which gives $1-t_{n-1}=t_1$ so that $t_1= \frac{1}{n},$ and $s$ is an $(n+1)$-valued discrete state.

For interest, we  also give  another proof. We are assuming $s(E)$ is a sub-effect algebra of $[0,1]$. Noticing that $\Gamma(\mathbb{R},1)=[0,1],$ and so, by  \cite[Thm 2.4]{BeFo97}, there exists a subgroup $G$ of $\mathbb{R}$ such that $\Gamma(G,1)=s(E)$. By  \cite[Lem 4.21]{Good86}, there are following two cases:

(a) $G$ is a dense subgroup of $\mathbb{R}.$ Then $|\Gamma(G,1)|=|s(E)|$ is infinite, which is a contradiction with our assumption.

(b) $G$ is a cyclic subalgebra of $\mathbb{R}$. Assume that $G$ is generated by a positive element $t,$ and so $G=\{nt\mid n\in \mathbb{Z}\}$.  Thus, by $\Gamma(\mathbb{R},1)\subseteq [0,1]$, we have that $t\in (0,1),$ and $nt=1.$  In fact, by $1\in G$, there exists a natural number $m$ such that $mt=1.$ Thus, we have that $\Gamma(G,1)=\{0,\frac{1}{m},\ldots,1\},$ which implies that $s(E)=\{0,\frac{1}{m},\ldots,1\}.$ However, $|s(E)|=n+1,$ and so $m=n,$ hence, $s(E)=\{0,\frac{1}{n},\ldots,1\}.$

Thus, we have proved that  $s(E)=\{0,\frac{1}{n},\ldots,1\}.$
\end{proof}

\section{Pseudo-effect algebras with two-valued  states}

In this section, we will study the  structure of  pseudo-effect algebras
with two-valued  states. We will prove that a pseudo-effect
algebra $E$ admits a two-valued  state if and only if
there exists an ideal $I$ such that $E=I\cup I^{-}=I\cup I^{\sim},$
where $I^{-}=\{i^{-}\mid i\in I\}$, $I^{\sim}=\{i^{\sim}\mid i\in I\}$ and
$I\cap I^{-}=I\cap I^{\sim}=\emptyset.$

We recall that a nonempty subset $I$ of a GPEA $E$ is called an {\it ideal} if the following conditions hold:

\begin{itemize}
\item[(i)] for any $a \in E$ and $i\in I$ with $a\leqslant i$, we have $a
\in I$;

\item[(ii)] for any $i, j \in I$ if  $i + j $ exists in $E$, then we
have $i + j \in I$.
\end{itemize}
If  a set $I$ is an ideal of a  PEA $E$ and $1\notin I,$ then the ideal $I$ is called {\it proper}.

An ideal $I$  in a GPEA $E$ is called  {\it normal} if, for any
$a,i,j\in E$  such that $a+ i$ and $j+ a$ exist and are
equal, we have $i\in I$ if and only if  $j\in I$.

For example, if $s$ is a state of a PEA $E$, then the  set $\Ker(s)=\{x\in E\mid s(x)=0\},$ {\it kernel} of $s$, is a normal ideal of $E.$

An ideal $I$  in a GPEA $E$ is called  {\it maximal} if it is a
proper ideal of $E$ and is  not included properly in any proper ideal of
$E.$

For example, the sets $\{0\}$ and $E$ are ideals of $E$. In
addition, if $I$ is an ideal of the GPEA $E$, then $(I;+,0)$ is also a
sub-GPEA of $(E;+,0).$

\begin{theorem}\label{th:sym}
Let  $(E;+,0,1)$ be a symmetric  PEA. The following two
statements are equivalent.
\begin{itemize}
\item[{\rm(i)}] There exists a two-valued  state on $E$.
\vspace{-2mm}
\item[{\rm (ii)}] There exists a sub-GPEA $(I;+,0)$ of $(E;+,0)$ such that
$E=\hat{I},$ $I$ is a maximal and normal  ideal of $E$.
\end{itemize}
\end{theorem}

\begin{proof}(i)$\Rightarrow$(ii). Assume that a  mapping $s:E\rightarrow
\{0,1\}$ is a state on $E$. If we set $I=\Ker(s),$ then $I\ne E$ is  a normal ideal of $E$, and so, it is  also  a sub-GPEA of $E$.

We now have to prove that $E=\hat{I}.$ Let $I^{\sharp}$ be the set
$\{a\in E\mid s(a)=1\}$. Since $s$ is a two-valued  state on
$E$, we have that  $I\cap I^{\sharp}=\emptyset$ and $E=I\cup
I^{\sharp}$. Define the mapping $f:\hat{I}\rightarrow E$ by
$f(x)=x,$ $f(x^{\sharp})=x^{\prime}$ for any $x\in I.$ Then $f$ is a
bijection and $f(0)=0,$ $f(0^{\sharp})=1.$ Assume  $a+ b$
exists in $\hat{I}$ for $a,$ $b\in \hat{I}.$ Then there are
following three cases. (1)\ Both $a$ and $b$ belong to $I$, then $
a+ b\in I,$ which implies $f(a+ b)=a+ b=f(a)+
f(b)$. (2)\ Only one of $a$ and $b$ belongs to $I,$ without loss of
generality, assume that $a\in I,$ $b\in I^{\sharp}.$ Then there
exists an element $c\in I$ with $b=c^{\sharp}.$ By $(c\backslash
a)+ a+ c$$^{\prime}=1,$ we have that $f(a+
b)=f(a+ c^{\sharp})=f((c\backslash a)^{\sharp})=(c\backslash
a)^{\prime}=a+ c^{\prime}=f(a)+ f(c^{\sharp})=f(a)+
f(b).$ Finally, (3)\ $a,b \in \hat I$ but this is impossible.  Hence, $f$ is a morphism. Furthermore, assume that
$f(a)\leqslant f(b)$ for $a,b\in \hat{I}.$ There are following four
cases. (1)\ If $a,$ $b\in I$, then
 $a\leqslant b$ by $f(a)=a,$ $f(b)=b.$ (2)\ If $a,$ $b\in I^{\sharp}$, then
there exist $c,$ $d\in I$ such that $a=c^{\sharp},$ $b=d^{\sharp},$
which imply that $c^{\prime}\leqslant d^{\prime}$, and so
$d\leqslant c.$ Hence, $a=c^{\sharp}\leqslant d^{\sharp}=b.$ (3)\ If
$a\in E,$ $b\in I^{\sharp},$ then there exists an element $c\in I$
with $b=c^{\sharp}.$ Therefore, $a\leqslant c^{\prime}$,  $a\leqslant
b$. (4) If $a\in I^{\sharp},$ $b\in I$,  then there exists an
element $d\in I$ with $a=d^{\sharp}.$  Therefore, $d^{\prime}\leqslant
b$, $d^{\sharp}\leqslant b$, which is impossible.  Hence, the statement $f(a)\leqslant f(b)$ implies that $a\leqslant b,$  which implies that $f$ is a monomorphism. Thus, $f$ is an isomorphism between $\hat{I}$ and $E$. Noticing that the set $\hat{I}$ equals $E$, we have the PEA  $\hat{I}$ coincides
with $E$.

Now, if there exists an ideal $J$ of $E$,
such that $I\subseteq J$ with $J\setminus I\neq\emptyset,$ then
there exists an element $i\in I$ such that $i^{\prime}\in J,$ hence,
$1\in J,$ which implies that $J=E.$

 (ii)$\Rightarrow$(i). Assume that PEA
$E=I\cup I^{\sharp},$ then $I$ is a symmetric GPEA by Proposition
\ref{pr:unitization}. Then define a mapping $s: E\rightarrow
\{0,1\}$ by setting $s(a)=0,$ $s(a^{\sharp})=1$ for any $a\in I$.
Consequently, $s(0)=0,$ $s(1)=s(0^{\sharp})=1.$ If $x,y\in E,$ and
$x+ y$ exists in $E,$ then  $x,y\in I$ or only exactly  one of
$x,y$ belongs to $I.$ If $x,y\in I,$ then $x+ y\in I$, and so
$s(x+ y)=s(x)+ s(y)=0.$ If $x\in I$ and $y\in I^{\sharp}$,
then $x+ y\in I^{\sharp}$, and so $s(x+ y)=1=s(x)+
s(y).$ Similarly, if $x\in I^{\sharp}$ and $y\in I$, then  $x+
y\in I^{\sharp}$, and so $s(x+ y)=1=s(x)+ s(y).$ Thus, $s$
is a two-valued  state on $E.$
\end{proof}

\begin{example}\label{ex:ex23}
{\rm Let $\mathbb{Z}$ be the group of integers and $G$ be a
po-group. Let $\mathbb{Z}\overrightarrow{\times}G$ be
the lexicographic product of $\mathbb{Z}$ and $G$. If we set
$E=\Gamma(\mathbb{Z}\overrightarrow{\times}G,(1,0)),$ then $E$ is a
symmetric PEA but not necessarily commutative. Set $I=\{(0,g)\in
E\mid g\in G\}$, then $I$ is a maximal and normal ideal of $E$ and
it is routine to verify that $E=I\cup I^{-}=I\cup I^{\sim}$ and
$I\cap I^{-}=I\cap I^{\sim}=\emptyset.$ Thus, the symmetric PEA $E$
admits a two-valued  state, and this state is a unique state of $E.$ }

\end{example}

In \cite{Rie08,RieMar05}, Z. Rie\v{c}anov\'{a} and I. Marinov\'{a}
studied   effect algebras with two-valued (discrete) states, and they proved that any effect algebra admitting a two-valued state is the unitization of a generalized sub-effect algebra.    Theorem
\ref{th:sym} shows that any symmetric PEA admitting a two-valued state is the unitization of a symmetric sub-GPEA, and so, it may be considered as a generalization of the results for
effect algebras proved in \cite{Rie08}.  However, for any PEA admitting a two-valued state, if it is not symmetric, then it is not a unitization of any sub-GPEA. In general, for a two-valued  state PEA, we have the following result.

\begin{theorem}\label{th:pea}
Let $E$ be a  PEA. Then $E$ admits   a two-valued  state
$s$ if and only if there exists a maximal and normal ideal $I$ such
that $E=I\cup I^{-}=I\cup I^{\sim}$ and $I\cap I^{-}=I\cap
I^{\sim}=\emptyset.$

\end{theorem}

\begin{proof}  Assume that  $E$ is a PEA admitting a  two-valued  state $s$,
then   for any $x\in E,$   either $s(x)=0$, or $s(x)=1.$ Set
$I=\Ker(s)$,  we have that    $E=I\cup(E\setminus I)$. For any $x\in
I$, $s(x^{-})=1,$ hence, we have that $I^{-}\subseteq E\setminus I.$
Conversely, for any $y\in E\setminus I,$ we have that $s(y)=1,$ and
so $s(y^{\sim})=0,$ which implies $y^{\sim}\in I.$ Noticing that
$y=y^{\sim-},$ we have that $y\in I^{-},$ and so $E\setminus
I\subseteq I^{-}.$ Hence, $E\setminus I=I^{-}.$ Similarly, we can
prove that $E\setminus I= I^{\sim}.$ It is obvious that $I\cap
I^{-}=I\cap I^{\sim}=\emptyset.$ In the same way as in Theorem \ref{th:sym}, we can prove that $I$ is a maximal and normal ideal of $E$.

Conversely, we assume that there exists a normal ideal $I$ such that
$E=I\cup I^{-}=I\cup I^{\sim}$ and $I\cap I^{-}=I\cap
I^{\sim}=\emptyset.$ Define a mapping $s:E\rightarrow \{0,1\}$ as
follows:

\begin{displaymath}
s(x)=\left\{
\begin{array}{ll}
0, & \textrm{$ x\in I$,}\\
1, & \textrm{otherwise.}\qquad\qquad
\end{array} \right.
\end{displaymath}

It is easy to see that $s$ is well defined and $s(0)=0,$ $s(1)=1.$
Now,  assume that  $x+ y$ exists in $E$ for   $x,y\in E,$ then
there are the following three cases:

(i)\ $x,y\in I$. Then  $x+ y\in I$, since $I$ is an ideal of
$E$. Therefore, $s(x+ y)=s(x)+ s(y)=0.$

(ii)\ Only one of $x$ and $y$ belongs to $I$; without loss of
generality, we assume that $x\in I$ and $y\notin I$. Then $x+
y\notin I$, since $I$ is an ideal of $E$. Consequently, $s(x+
y)=s(x)+ s(y)=1.$

(iii)\ $x\notin I$ and $y\notin I.$ Now we assume that there exist
$a,b\in I$ with $x=a^{-}$ and $y=b^{-}.$ Then $a^{-}+ b^{-}$
exists, which implies $b^{-}\leqslant a^{-}$$^{\sim}=a$. Hence,
$y=b^{-}\in I,$ which is a contradiction with $y\notin I.$ Thus, if
$x+ y$ exists in $E$, then at least one of  $x $ and $y$
belongs to $I.$

Hence, we have proved that  for $x,y\in E,$ $s(x+ y)=s(x)+
s(y),$ whenever $x+ y$ exists in $E$. This yields that the mapping $s$
is a two-valued  state on $E.$
\end{proof}

\begin{example}\label{thm:level1}
{\rm Let  $E$ be the PEA $\Gamma(G,(1,0,0))$ in Example
\ref{ex:ex21}. Assume that  $s:E\rightarrow [0,1]$ is a state on
$E$. Notice that, for any $(0,b,c)\in E,$ $n(0,b,c)$ exists in $E$ for
any $n\in \mathbb{N}$. Hence, we have $s(0,b,c)=0.$  Then
$\Ker(s)=\{(0,b,c)|(0,b,c)\in E\}$ which is a normal ideal of $E$.
Furthermore,  for any $(0,b,c)\in E,$ it is easy to see that
$(0,b,c)^{-}= (1,-b,-c)),$ $(0,b,c)^{\sim}= (1,-c,-b)),$ which
implies that
$E=\Ker(s)\cup(\Ker(s))^{-}=\Ker(s)\cup(\Ker(s))^{\sim}$ and $
\Ker(s)\cap(\Ker(s))^{-}=\Ker(s)\cap(\Ker(s))^{\sim}=\emptyset.$ Therefore,
the state $s$ of $E$ is  two-valued, and this state is a unique state of $E.$}
\end{example}

\section{Pseudo-effect algebras with $(n+1)$-valued discrete states}

In this section,   we give   sufficient and necessary
conditions in order  a pseudo-effect algebra admits an $(n+1)$-valued state. In addition, some properties of  pseudo-effect algebras having an $(n+1)$-valued state are studied.

Let $E$ be a PEA and $A,B\subseteq E.$ In the following, we write (i)
$A\leqslant B$ iff $a\leqslant b$ for all $a\in A,$ and all $b\in B,$ (ii) $A+B:=\{a+b\mid  a \in A,\ b \in B$ and $a+b$ exists in $E\}.$ It can happen that $a+b$ exists in $E$ for any $a\in A$ and any $b\in B.$ Then we are saying that $A+B$ {\it exists in} $E$.

We write $1A:=A.$ If $A+ A$ exists, then
we denote $2A=A+ A.$ If $iA$ exists, and $iA+ A$ exits,
then we denote $(i+1)A=iA+ A$ for $i\geqslant 2.$

\begin{theorem}\label{th:nstate}
Let  $(E;+,0,1)$ be a  PEA. Then the following two statements
are equivalent.

{\rm (i)} There exists an $(n+1)$-valued discrete state on $E$.

{\rm (ii)} There exist nonempty  subsets $E_{0}, E_{1},\ldots,
E_{n}$ of
 $E$ such that

{\rm (a)} $E_{i}\cap E_{j}=\emptyset,$ for any $i,j\in
\{0,1,\ldots,n\}$ with $i\neq j,$

{\rm (b)} $E=E_{0}\cup E_{1}\cup\cdots\cup E_{n},$

{\rm (c)} $E_{i}^{-}=E_{i}^{\sim}=E_{n-i}$ for any $i\in
\{0,1,\ldots,n\},$

{\rm (d)} if $x\in E_{i},$ $y\in E_{j}$ and $x+ y$ exists in
$E$, then $i+j\leqslant n$ and $x+ y\in E_{i+j}$ for
$i,j\in\{0,1,\ldots,n\}.$
\end{theorem}

\begin{proof}
Assume that $s$ is an $(n+1)$-valued discrete state on $E,$ then we
set $E_{i}=s^{-1}(\{\frac{i}{n}\})$ for any $i\in \{0,1,\ldots,n\}.$
It is easy to see that (a) and (b) hold. For (c),  $x\in E_{i}$
if and only if  $s(x)=\frac{i}{n}$ if and only if
$s(x^{-})=s(x^{\sim})=\frac{n-i}{n}$, which entails  that the
statement (c) holds. For (d), assume that $x\in E_{i},$ $y\in E_{j}$
and $x+ y$ exists in $E$, then we have that $s(x)=\frac{i}{n}$,
$s(y)=\frac{j}{n}$ and $s(x+
y)=s(x)+s(y)=\frac{i+j}{n}\leqslant1,$ which implies that
$i+j\leqslant n$ and $x+ y\in E_{i+j}.$

Conversely, define a mapping $s:E\rightarrow [0,1]$ by
$s(x)=\frac{i}{n}$ if $x\in E_{i}.$ It is clear that $s$ is
well-defined and $s(E)=\{0,\frac{1}{n},\ldots,\frac{n-1}{n},1\}$.
Take $x,y\in E$ such that $x+ y$ is defined in $E.$ Then there
are unique integers $i$ and $j$ such that $x\in E_{i}$ and $y\in
E_{j}$. By (d), we have that $i+j\leqslant n$ and $x+ y\in
E_{i+j}$. Hence, $s(x+ y)=s(x)+s(y).$ Furthermore, there is a
unique $i\in\{0,1,\ldots,n\}$ such that $0\in E_{i}$. For any $x\in
E_{n},$ $x+ 0$ and $0+ x$ exist, and so, by (d)
$i+n\leqslant n,$ which implies $i=0.$ Thus, $0\in E_{0}$ and $1\in
E_{n}$ by (c). Hence, $s(0)=0,$ and $s(1)=1.$ Thus, $s$ is an
$(n+1)$-valued discrete state on $E$.
\end{proof}

Let $E$ be a PEA and $n\ge $ be an integer. If
subsets $E_{0},\ldots,E_{n}$ of $E$ satisfy the conditions (a)--(d)
in Theorem \ref{th:nstate},  we say that $E_{0},\ldots,E_{n}$ is
an {\it $n$-decomposition} of $E$ and  we shall denote it
by $(E_{0},\ldots,E_{n}).$ Let $\mathcal{D}_{n}(E)=\{(E_{0},\ldots,E_{n})\mid$
$(E_{0},\ldots,E_{n})$ is an $n$-decomposition of $E\}$ and
$\mathcal{S}_{n}(E)=\{s\mid s $ is an $(n+1)$-valued discrete state
on $E\}$.

\begin{theorem}\label{th:uniq}
 Let  $(E;+,0,1)$ be a  PEA and $n\ge 1$ be an integer. Then there is a bijective mapping between
$\mathcal{D}_{n}(E)$ and $\mathcal{S}_{n}(E)$.
\end{theorem}

\begin{proof}
We define a mapping
$f:\mathcal{D}_{n}(E)\rightarrow\mathcal{S}_{n}(E)$ as follows: for
any $D=(E_{0},\ldots,E_{n})\in \mathcal{D}_{n}(E),$ $f(D)=s,$ where
$s:E\rightarrow [0,1]$ is a state such that $s(E_{i})=\frac{i}{n}$
for any $i\in \{0,\ldots,n\}.$ Assume that there exists another
state $s_{1}$ on $E$ such that $s_{1}(E_{i})=\frac{i}{n}$ for any
$i\in \{0,\ldots,n\}.$  For any $x\in E,$ there exists a unique
$i\in  \{0,\ldots,n\}$ such that $x\in E_{i}$ which implies that
$s(x)=s_{1}(x).$ Thus, $f$ is defined well.  Now, for any $D=(E_{0},
\ldots, E_{n}),$ and $D_{1}=(F_{0}, \ldots, F_{n})$, if
$f(D)=f(D_{1})=s$, then $s(E_{i})=s(F_{i})=\frac{i}{n}$ for any
$i\in\{0,\ldots,n\}.$ Hence, $s^{-1}(\{\frac{i}{n})\}=E_{i}=F_{i}$ for
any $i\in\{0,\ldots,n\},$ and so $D=D_{1},$ which implies that $f$ is  injective.  By Theorem \ref{th:nstate}, $f$ is surjective. Thus, $f$ is bijective.
\end{proof}

\begin{corollary}\label{co:ideal}
Let  $(E;+,0,1)$ be a  PEA. If $(E_0,E_1,\ldots,E_n)$ is an
$n$-decomposition of $E$, then $E_{0}$ is a
normal ideal.
\end{corollary}

\begin{proof}
By Theorem \ref{th:nstate}, there exists an $(n+1)$-valued discrete
state $s$ such that $E_{0}=\Ker(s)$, and so it is a normal ideal.
\end{proof}

\begin{remark}\label{re:nounique}
{\rm Assume that a PEA $(E;+,0,1)$ admits an $(n+1)$-valued discrete state, $s, $ then by Theorem \ref{th:uniq}, there exists a unique $n$-decomposition  $(E_{0}, E_{1},\ldots, E_{n})$ of $E$ such that $s(E_i)=\frac{i}{n},$ $i=0,1,\ldots,n.$  We note that:

(i) For $i,j\in \{0,1,\ldots,n\}$ with $i\leqslant j,$
$E_{i}\leqslant E_{j}$ does not hold in general. For example, the four element Boolean algebra $E=\{0,a,a^{\prime},1\}$ admits a 2-valued discrete state $s$ such
that $s(0)=s(a)=0, s(a^{\prime})=s(1)=1$. Set $E_{0}=\{0,a\}$, $E_{1}=\{0,a^{\prime}\}$, then  $E_{0}\not\leqslant E_{1}.$

(ii) In general, for $i,j\in \{0,1,\ldots,n\}$,   $E_{i}+ E_{j}$ does not
exist when $i+ j<n$. Even $E_{0}+ E_{0}$ does not exist.  For example, the four element Boolean algebra $E=\{0,a,a^{\prime},1\}$ admits a two-valued  state $s$ such
that $s(0)=s(a)=0, s(a^{\prime})=s(1)=1$. Set  $E_{0}=\{0,a\}$, $E_{1}=\{0,a^{\prime}\}$, then  $E_{0}+ E_{0}$ does not exists in $E.$

(iii) By Theorem \ref{th:pea}, if $n=1,$ then  the ideal $E_{0}$ is maximal. However, if $n\geqslant2,$ then $E_{0}$ is not necessarily  maximal. For example, the four element Boolean algebra $E=\{0,a,a^{\prime},1\}$ admits a 3-valued discrete state $s$ such
that $s(0)=0, s(a)=s(a^{\prime})=\frac{1}{2},s(1)=1$. But
$E_{0}=\{0\}$ is not a maximal ideal of $E.$
 }
\end{remark}

\begin{theorem}\label{th:infinit}
Let  $(E_{0}, E_{1},\ldots, E_{n})$ be an $n$-decomposition of a PEA $E$. Then $E_{0}\leqslant E_{1}\leqslant\cdots\leqslant E_{n}$ if and only if $E_{i}+ E_{j}$ exists in $E$
whenever $i+j<n$ for any $i,j\in \{0,\ldots,n\}.$

In such a case,
\begin{itemize}

\item[{\rm (i)}] $E_{0}=\mbox{\rm Infinit}(E)$ and $\mbox{\rm Infinit}(E)$ is a normal ideal.

\vspace{-2mm}
\item[{\rm (ii)}] $E_{i}+ E_{j}=E_{i+j}$ whenever $i+j<n$.

\vspace{-2mm}
\item[{\rm (iii)}] For any $x\in E_{i},$ $y\in E_{j}=E_{i+j},$ if $i+j>n$,
then  neither $x+ y$ nor $y+ x$ exists.
\end{itemize}

\end{theorem}

\begin{proof} By Theorem \ref{th:nstate}, there is a unique discrete $(n+1)$-valued state $s$ such that $s(E_i)=\frac{i}{n}$ for $i=0,1,\ldots,n.$

Assume $E_{0}\leqslant E_{1}\leqslant\cdots\leqslant E_{n}$. For any $i,j\in \{0,\ldots,n-1\}$ with $i+j<n,$ we have that $i<n-j$, and so $E_{i}\leqslant E_{j}^{-},$ which implies that $E_{i}+
E_{j}$ exists and so $E_{i}+ E_{j}=E_{i+j}.$ In fact, for any $a\in E_{i},$ $b\in E_{j},$  then $s(a+ b)=\frac{i+j}{n}$,
which implies that $a+ b\in E_{i+j}.$ Conversely, let $c\in E_{i+j}.$ For any $a\in E_{i}$, we have that $a\leqslant c.$ Then there exists an element  $b\in E$ such that $a+ b=c.$ Whence, $s(a+ b)=s(a)+s(b)=\frac{i+j}{n},$ then $s(b)=\frac{j}{n}$, which implies that $b\in E_{j}.$ We have also proved (ii).

Conversely, let $E_i + E_j$ exist in $E$ for $i+j<n.$

(i) For any $x, y\in E_{0},$ we have that $x+ y$ exists in $E.$  Then $s(x+y)=s(x)+s(y)=0$ and $x+y \in E_0$ which implies
$E_{0}\subseteq \mbox{\rm Infinit}(E).$
Conversely, let $x\in \mbox{\rm Infinit}(E),$ we have that $mx$ is defined in $E$ for each integer $m\ge 1.$ Then $s(mx)=ms(x) \leqslant 1$ which implies $s(x)=0$ and $x\in \Ker(s),$ and so $x\in E_{0}.$

For  $i,j\in \{0,1,\ldots,n-1\}$, if $i+j< n$, then
$E_{i}+ E_{j}$ exists in $E$, and so  $E_{i}\leqslant
E_{j}^{-}=E_{n-j}.$ Now, for $i\in \{0,1,\ldots,n-1\},$ set
$j=n-i-1$, we have that $i+j< n$, and so we have that
$E_{i}\leqslant E_{n-j}^{-}=E_{i+1},$ which proves $E_{0}\leqslant E_{1}\leqslant\cdots\leqslant E_{n}.$

(iii) Assume that $a\in E_{i}$ and $b\in E_{j}$ for $i+j<n.$ Then
$a+ b$ exists and $s(a+ b)= \frac{i+j}{n},$ and so
$a+ b\in E_{i+j}.$ Conversely, let $z\in E_{i+j},$ then for any
$x\in E_{i}$, $x\leqslant z$, so that $z=x+(x/z),$ by
$s(z)=s(x)+s(x/z),$ which implies that $x/z\in E_{j}.$

(iv) Assume that $i+j>n$, $x\in E_{i},$ $y\in E_{j}=E_{i+j},$ either
$x+ y$  or $y+ x$ exists, then $s(x+ y)>1$ or
$s(y+ x)>1$, which is absurd.
\end{proof}

\begin{example}\label{ex:eperfect}
{\rm  Let $D$ be the set $\{0,a,b,1\}$. Let a partial operation
$+_{D}$ on $B$ be defined  as follows:
$a+_{D} a=b+_{D} b=1,$ $0+_{D} a=a+_{D} 0=a,$
$0+_{D} b=b+_{D} 0=b,$ $1+_{D} 0=0+_{B} 1=1.$
Then the algebraic system $(D;+_{D},0,1)$ is an effect algebra, which is usually called the diamond.
Let $E_{0}=\{(0,i)\mid i\in \mathbb{Z}^{+}\}$, $E_{1}=\{(a,i)\mid
i\in \mathbb{Z} \}\cup \{(b,j)\mid j\in \mathbb{Z}\}$,
$E_{2}=\{(1,-i)\mid i\in \mathbb{Z}^{+}\},$ and $E=E_{0}\cup E_{1}
\cup E_{2}.$   We define a partial binary operation $+$ on
$E$ as follows:

(i) for any $x=(0,i),$ $y=(0,j)\in E_{0},$ $x+ y=(0,i+j).$

(ii) for any $x=(0,i)\in E_{0},$ $y=(a,j)\in E_{1},$ then $x+
y=y+ x=(a,i+j).$
For any $x=(0,i)\in E_{0},$ $z=(b,j)\in E_{1},$ then $x+ z=z+ x=(b,i+j).$

It is routine to verify that $(E;+,0,1)$ is an effect algebra,
where $0$ and $1$ denote $(0,0)$ and $(1,0)$, respectively. A
mapping $s:E\rightarrow [0,1]$  such that $s(E_{i})=\frac{i}{2}$ for $i=0,1,2$ is a 3-valued discrete state.

The following statements are true.

(1) $E_{0}=E_{0}+ E_{0},$ $E_{1}=E_{0}+ E_{1}.$

(2) Any of the following sum   $E_{0}+
E_{2}$, $E_{1}+ E_{1}$, $E_{1}+ E_{2}$ does not exist.

(3) $E_{0}\leqslant E_{1}\leqslant E_{2}.$

(4)  $E_{0}=\mbox{\rm Infinit}(E)$ and $E_{0}$ is  a maximal ideal.
  }
\end{example}

\begin{example}\label{ex:eperfect1}
{\rm Let $B$ be the set $\{0,a,b,1\}$. Let a partial operation
$+_{B}$ on $B$ be defined as follows:
$a+_{B} b=b+_{B} a=1,$ $0+_{B} a=a+_{B} 0=a,$
$0+_{B} b=b+_{B} 0=b,$ $1+_{B} 0=0+_{B} 1=1.$
Then the algebraic system $(B;+_{B},0,1)$ is an effect algebra.
Let $(G,u)$ be a  po-group with strong unit $u$. Let
$E_{0}=\{(0,i)\mid i\in G^{+}\}$, $E_{1}=\{(a,i)\mid i\in G \}\cup
\{(b,j)\mid j\in G\}$, $E_{2}=\{(1,-i)\mid i\in G^{+}\},$ and
$E=E_{0}\cup E_{1} \cup E_{2}.$   We define a partial binary
operation $+$ on $E$ as follows:

\begin{itemize}
\item[(i)] for any $x=(0,i), y=(0,j)\in E_{0},$ $x+ y$ exists and
$x+ y=(0,i+j),$

\vspace{-2mm}\item[(ii)] for any $x=(a,i),$ $y=(b,j)\in E_{1},$  if $i+j\leqslant 0$,
then $x+ y$ exists and  $x+ y=(1,i+j),$

\vspace{-2mm}\item[(iii)] for any $x=(0,i)\in E_{0},$ $y=(a,i),$ $z=(b,j)\in E_{1},$
$x+ y,$ $y+ x,$ $x+ z,$ and  $z+ x $ exist, and
$x+ y=(a,i+j),$ $y+ x=(a,j+i),$ $x+ z=(b,j),$
$z+ x=(b,j+i).$
\end{itemize}
It is routine to verify that $(E;+,0,1)$ is a PEA, where $0 $
and  $1$ denote $(0,0)$ and $(1,0)$, respectively.

It is easy to see that $E_{i}+ E_{j}$ exists for $i+j<2.$

We have $E_{0}=\mbox{\rm Infinit}(E),$ however, $E_{0}$ is not a maximal ideal. If
we set $I_{a}=E_{0}\cup\{(a,i)\mid (a,i)\in E_{1}\}$ and
$I_{b}=E_{0}\cup\{(b,j)\mid (b,j)\in E_{1}\}$, then both $I_{a}$ and
$I_{b}$ are   proper normal ideals, and $E_{0}\subsetneq I_{a},$ $
E_{0}\subsetneq I_{b}.$ In fact, $\{I_{a},I_{b}\}$ is the set of
maximal ideals of $E,$ and $E_{0}=I_{a}\cap I_{b}.$ }
\end{example}

\section{$n$-perfect  PEA}

We now give the definition of $n$-perfect PEAs as follows.

\begin{definition}\label{de:npefect}
{\rm Let  $(E;+,0,1)$ be a  PEA.  We say that $E$ is an  {\it
$n$-perfect} PEA if
\begin{itemize}
\item[(i)] there exists an $n$-decomposition $(E_{0}, E_{1},\ldots, E_{n})$
of $E$.

\item[(ii)] $E_{i}+ E_{j}$ exists   if $i+j<n$.

\item[(iii)] $E_{0}$ is the unique maximal ideal of $E$.
\end{itemize}

}
\end{definition}

We recall that according to Corollary \ref{co:ideal}, $E_0$ is a unique maximal ideal of $E$ and it is normal.

\begin{example}\label{ex:nperf}
{\rm Let $\mathbb{Z}$ be the group of integers and $G$ be a
po-group. Let $\mathbb{Z}\overrightarrow{\times}G$ be
the lexicographic product of $\mathbb{Z}$ and $G$, and let
$u=(n,0)$. If we set
$E=\Gamma(\mathbb{Z}\overrightarrow{\times}G,(n,0)),$ then $E$ is a
PEA. If we set $E_{0}=\{(0,g)\mid g\in G^{+}\}$, for $i\in
\{1,\ldots,n-1\}$, $E_{i}=\{(i,g)\mid g\in G\}$, and
$E_{n}=\{(0,-g)\mid g\in G^{+}\}$, then $E$ is an $n$-perfect PEA.}
\end{example}

We recall the following two definitions used in \cite{Dv08}. Let $E$ be
a PEA. We denote by $\mathcal{M}(E)$ and $\mathcal{N}(E)$ the set of
maximal ideals and the set of normal ideals of $E,$ respectively. We define (i) the {\it radical} of a PEA $E$, $\Rad(E)$, as the set
$$\Rad(E)=\bigcap\{I\mid I\in \mathcal{M}(E)\},$$
and (ii) the {\it normal radical} of $E$, via
$$\Rad_{n}(E)=\bigcap\{I\mid I\in \mathcal{M}(E)\cap\mathcal{N}(E)\}.
$$
It is obvious that $\Rad(E)\subseteq \Rad_{n}(E)$ holds in any PEA
$E.$

\begin{lemma}\label{le:radeq}
{\rm Let  $(E;+,0,1)$ be an $n$-perfect PEA.  Then
$E_{0}=\mbox{\rm Infinit}(E)=\Rad(E)=\Rad_{n}(E).$ }
\end{lemma}

\begin{proof} By Theorem \ref{th:infinit}, $E_{0}=\mbox{\rm Infinit}(E)$. By
(iii) of Definition \ref{de:npefect}, we have that $\Rad(E)=E_{0}.$
Furthermore, $E_{0}$ is also a normal ideal and so
$\Rad(E)=\Rad_{n}(E).$
\end{proof}

\begin{definition}\label{de:rieszid}

{\rm An ideal $I$ in a GPEA $E$ is called an $R_{1}$-$ideal,$ if the
following condition holds:

(R1) if $i\in I,a,b\in E$ and  $a+ b$ exists, $i\leqslant
a+ b,$ then there exist $j,k\in I$ such that $ j \leqslant a,$
$k\leqslant b$ and $i\leqslant j+ k$.

An  $R_{1}$-ideal $I$ is called a $Riesz$ $ideal,$ if the following two
conditions  hold:

(R2) if $i\in I,a,b\in E$,  $i\leqslant a$  and $(a\backslash
i)+ b$ exists, then there exists  $j\in I$ such that
$j\leqslant b $ and $a+ (j/b)$ exists; if $i\in I,a,b\in
E,i\leqslant a$ and $b+(i/a)$ exists, then there exists  $j\in
I$ such that $j\leqslant b $ and $( b\backslash j)+ a$ exists.}

\end{definition}

Let $A$ be a subset of a partially ordered set $E.$ We say that
$A$ is {\it downwards} ({\it upwards}) directed if for any $x,y\in A,$ there exists $z\in A$ such that $z\leqslant x,y$ $(x,y\leqslant z)$. If $E$ is a po-group or a PEA, then $E$ is upwards directed iff it is downwards directed; then we say simply that $E$ is {\it directed}.

\begin{proposition}\label{pr:rideal}

{\rm \cite{XieLi10a} } In an upwards directed GPEA $E$, an ideal $I$
is a Riesz ideal if and only  if  $I$ is $R_{1}$-ideal.

\end{proposition}

\begin{proposition}\label{pr:radeq}
Let  $(E;+,0,1)$ be an $n$-perfect PEA for some integer $n\ge 1.$  Then $E_{0}$ is a
Riesz ideal.
\end{proposition}

\begin{proof} Since the PEA $E$ is upwards directed, by Proposition \ref{pr:rideal}, it suffices to show that $E_{0}$ satisfies the
condition (R1). Assume that $i\in E_{0},$ $a, b\in E,$ $a+ b$
exists,  and $i\leqslant a+ b.$ There are the following three
cases: (1) Both $a$ and $b$ belong to $E_{0};$ then the condition is trivial. (2)
Only one of $a, b$ belongs to $ E_{0}$, without loss of generality,
we assume that $a\in E_{0}$ and $b\notin E_{0}.$ By Theorem
\ref{th:infinit} (ii), $i\leqslant b$, thus $i\leqslant a+
i\leqslant a+ b.$ (3) Neither $a$ nor $ b$ belongs to $E_{0}$,
then by Theorem \ref{th:infinit} (ii), $i\leqslant a, b$, and so
$i\leqslant i+ i\leqslant a+ b$, and $i\in E_{0}.$
\end{proof}

\begin{definition}\label{def:t-norm1}
{\rm
 For an ideal $I$ in a GPEA $E$, we define $a\sim_{I}b$ if there
exist $i, j\in I,$ $i\leqslant a,$ $j\leqslant b$ such that
$a\setminus i=b\setminus j.$ }
\end{definition}

\begin{theorem}\label{th:linear}
{\rm \cite{XieLi10a} Let  $I$ be  a normal Riesz ideal  in a GPEA
$E$. Then $E/\sim_{I}$ is a linear GPEA if and only if $I$ satisfies
the following condition:

 (L) For any $a,b\in E$, there exists a $c\in E$ such that  $a+ c  \sim_{I} b$ or $b+ c  \sim_{I} a.$ }
\end{theorem}

\begin{lemma}\label{le:direct}
{\rm \cite{XieLi10a}} If $\sim$ is  a Riesz congruence  in a  PEA
$E,$ then for any $a\in E$ the equivalence class $[a]$ is both
upwards directed and downwards directed.
\end{lemma}

\begin{lemma}\label{le:dir}
{\rm \cite{Dv03}} If $I$ is  an ideal of PEA $E$ satisfying
{\rm (RDP)}$_{0}$ and  $a$ is an element of $E,$ then the ideal $I_{0}(I,a)$
generated by $I$ and  $a$ is given by $I_{0}(I,a)=\{x\in E\mid
x=x_{1}+ a_{1}+\cdots + x_{n}+ a_{n},\ x_{i}\in
I, a_{i}\leqslant a, 1\leqslant i\leqslant n,\ n\geqslant1\}.$ If $I$
is a normal ideal, then $I_{0}(I,a)=\{x\in E\mid x=x_{1}+
a_{1}+\cdots+
 a_{n},\ y\in I, a_{i}\leqslant a, 1\leqslant
i\leqslant n,\ n\geqslant1\}.$
\end{lemma}

\begin{proposition}\label{pr:1perf}
Let $(E;+,0,1)$ be an $n$-perfect PEA. Then  $E_{0}$ and $E_{n}$ are both  upwards  and downwards directed.
\end{proposition}

\begin{proof} By $0\in E_{0},$ $E_{0}$ is  downwards directed. For any $x,y\in
E_{0},$ $x+ y$ exists in $E$ and $x+ y\in E_{0},$ thus $E_{0}$ is  upwards directed. By $E_{n}=E_{0}^{\sim}$, we have that $E_{n}$ is both upwards directed and downwards directed.
\end{proof}

\begin{proposition}\label{pr:uddirect}
Let  $(E;+,0,1)$ be an $n$-perfect PEA satisfying  {\rm (RDP)}$_{0}$. Then $E$ satisfies the following condition

{\rm(e)} for any $i\in \{0,1,\ldots,n\},$ $E_{i}$ is both upwards
directed and downwards directed.
\end{proposition}

\begin{proof}
By Proposition \ref{pr:1perf}, if $n=1,$ then the result holds. Now,
assume that $n>1.$  By $E_{i}=E_{i}^{-}=E_{i}^{\sim}$, it suffices
to prove that $E_{i}$ is downwards directed, for any $i\in
\{1,\ldots,n-1\}$.

Assume that $x,y\in E_{1}$. Then the ideal $I(E_{0}, x)$, which is
generated by the normal ideal $E_{0}$ and $x$, is equal to $E$,
since $E_{0}$ is a maximal ideal. Thus, there exists $a\in E_{0}, $
and $z_{1},\ldots,z_{m}\in E_{1}$ with $z_{1},\ldots,z_{m}\leqslant
x$ such that $y=a+ z_{1}+\cdots + z_{m}$ by Lemma
\ref{le:dir}. By $y\in E_{1},$ we have that $m=1$. Thus,
$z_{1}\leqslant x, y$ and $z_{1}\in  E_{1}$ and
$E_{1}$ is downwards directed.

By Theorem \ref{th:infinit}, $E_{i}=iE_{1}$ for any $i\in
\{1,\ldots,n-1\}.$ Now, assume that $x,y\in  E_{i},$ then there
exist  $x_{1},\ldots,x_{i}\in E_{1}, $  and $y_{1},\ldots, y_{i}\in
E_{1},$ such that $x=x_{1}+\cdots + x_{i}$ and
$y=y_{1}+\cdots + y_{i}.$ Since $E_{1}$ is downwards directed,
there exists a $z\in E_{1}$ such that $z\leqslant x_{1},\ldots,x_{i}
$ and $z\leqslant y_{1},\ldots, y_{i}.$ Thus, $iz\in E_{i}$ and
$iz\leqslant x,y. $ Hence, $E_{i}$ is downwards directed for any
$i\in\{1,\ldots,n-1\}.$
\end{proof}

For $a,b \in E$ with $a\leqslant b,$ we define an interval $[a,b]:=\{x\in E \mid a\leqslant x\leqslant b\}.$

\begin{proposition}\label{pr:minimum}
Let $(E;+,0,1)$ be an $n$-perfect PEA satisfying the condition
{\rm(e)}. If any decreasing chain in $E_{1}$ has a lower bound
in $E_{1}$, then
\begin{itemize}
\item[{\rm (i)}] There exists a smallest element $c\in E_{1}$.
\vspace{-2mm}
\item[{\rm (ii)}] For any $i\in \{0,1,\ldots,n\},$ $ic$ is a smallest
element in $E_{i}$.

\vspace{-2mm}\item[{\rm (iii)}] For any $i\in \{0,1,\ldots,n\},$ $(ic)^{\sim}=(ic)^{-}$
and it is the largest element  in $E_{n-i}.$
\vspace{-2mm}
\item[{\rm (iv)}] For any $i\in \{0,1,\ldots,n\},$
 $E_{i}=[ic,((n-i)c)^{\sim}].$

\vspace{-2mm}
\item[{\rm (v)}] $E=\{0,c,\ldots,nc\}.$

\end{itemize}
\end{proposition}

\begin{proof}  By Zorn's Lemma, there exists a minimal element $c$
in $E_{1}$. Since $E_{1}$ is downwards directed, and so the minimal
element $c$ is also the smallest element in $E_{1}$.

Now, by  Theorem \ref{th:infinit} (iii), for any $i\in
\{1,\ldots,n\},$ $E_{i}=iE_{1}.$ For $i\in \{1,\ldots,n\},$ $x\in
E_{i},$ there exist $x_{1},\ldots, x_{i}\in E_{1}$ such that
$x=x_{1}+\cdots+ x_{i},$ which implies that $ic\leqslant
x,$ since $c$ is  the smallest element in $E_{1}$. Hence, for any
$i\in \{0,1,\ldots,n\},$ $ic$ is the smallest element in $E_{i}$.
Thus, $(ic)^{\sim}$ and $(ic)^{-}$ are the largest elements in $E_{n-i},$
and so  $(ic)^{\sim}=(ic)^{-}.$ Hence,
$E_{i}=[ic,((n-i)c)^{\sim}],$ for any $i\in \{0,1,\ldots,n\}.$

Since $(nc)^{-}$ is the largest element of $E_{0}$,
$(nc)^{-}+ (nc)^{-}$ exists and $(nc)^{-}+ (nc)^{-}\in
E_{0},$ then we have that $(nc)^{-}=0,$ which implies that $nc=1.$
Thus, we have that $ic=((n-i)c))^{\sim}$, for $i\in
\{0,1,\ldots,n\}$. By $E_{i}=[ic,((n-i)c)^{\sim}],$ we have that
$E_{i}=\{ic\},$ $i\in \{0,1,\ldots,n\}$.
\end{proof}

Recall that for any state $s:E\rightarrow [0,1]$ on a PEA $E$, we can
define a binary operation $\sim_{s}$ as follows: $x\sim_{s}y$ if
and only if $s(x)=s(y)$, for $x,y\in E$.

\begin{proposition}\label{pr:equcladd}
Let  $(E;+,0,1)$ be an $n$-perfect PEA satisfying the
condition {\rm(e)} and let $s$ be a state $s:E\rightarrow[0,1]$ such that
$s(E_{i})=\frac{i}{n}$ for any $i\in\{0,1,\ldots,n\}$. Then:

\begin{itemize}
\item[{\rm (i)}] For $x,y\in E$, $x\sim_{s}y$  if only if there exists a
unique $i\in \{0,1,\ldots,n\}$  such that $x,y\in E_{i}.$

\vspace{-2mm}
\item[{\rm(ii)}] For $x,y\in E$, $x\sim_{s}y$  if only if $x\sim_{E_{0}}y.$
\end{itemize}

\end{proposition}

\begin{proof}
(i) It is obvious.

(ii) For $x,y\in E$,  if $x\sim_{E_{0}}y$, then there exist $a,b\in
E_{0}$ such that $x\setminus a=y\setminus b$, thus $s(x\setminus
a)=s(y\setminus b).$ By $s(E_{0})=0$, we have that  $s(x)=s(y),$ and
so $x\sim_{s}y$.

Conversely, if $x\sim_{s}y$, then there is a unique $i\in
\{0,1,\ldots,n\}$  such that $x,y\in E_{i}.$ Now, if $i=0,$ then
$x\sim_{E_{0}}y.$ If $i=n,$ then $1\setminus x,1\setminus y\in
E_{0},$ and so $1\setminus x \sim_{E_{0}}1\setminus y,$ hence,
$x\sim_{E_{0}}y.$ If $n=1$, we have finished the proof. Assume that
$n>1$ and $i\in\{1,\ldots,n-1\}.$ Since $E_{i}$ is downwards
directed, there exists $z\in E_{i}$ such that $z\leqslant x,y.$
Thus, $z=x\backslash(z/x)=y\backslash(z/y)$ and $z/x, z/y\in E_{0},$
which implies that $x\sim_{E_{0}}y$.
\end{proof}

\begin{proposition}\label{pr:caequ}
Let  $(E;+,0,1)$ be an $n$-perfect PEA and a state
$s:E\rightarrow[0,1]$ such that $s(E_{i})=\frac{i}{n}$ for any
$i\in\{0,1,\ldots,n\}$.    If for $x,y\in E$, $x\sim_{s}y$  if only
if $x\sim_{E_{0}}y,$ then $E$ satisfies the condition {\rm(e)}.
\end{proposition}

\begin{proof}
By the assumption, we have that  for any $i\in \{0,1,\ldots,n\}$,
$E_{i}$ is the equivalent class of $E$ with respect to the Riesz congruence
$\sim_{E_{0}}.$ Thus, by Lemma \ref{le:direct}, for any $i\in
\{0,1,\ldots,n\}$, $E_{i}$ is both upwards directed and downwards
directed.
\end{proof}

\begin{theorem}\label{th:nperfect}
 Let  $(E;+,0,1)$ be an $n$-perfect PEA satisfying the
condition {\rm(e)}. Then $E/\sim_{E_{0}}$ is isomorphic to the
effect algebra $\{0,\frac{1}{n},\ldots,\frac{n-1}{n},1\}.$
\end{theorem}

\begin{proof}
For any $x,y\in E,$ there exist unique $i,j\in\{0,1,\ldots,n\}$ such
that $x\in E_{i},$ $y\in E_{j}.$ Without loss of generality, we
assume that $i\leqslant j.$ Assume that $i=j.$ By Theorem
\ref{th:nstate}, there exists a state
$s:E\rightarrow\{0,\frac{1}{n},\ldots,\frac{n-1}{n},1\}$ such that
$s(E_{i})=\frac{i}{n}$, for $i\in \{0,1,\ldots,n\}.$ By Proposition
\ref{pr:equcladd}, $E/\sim_{E_{0}}=\{E_{0}, E_{1},\ldots, E_{n}\},$
then $x\sim_{E_{0}}y$. If $i< j,$ then $x< y,$ and so there
exists $z\in E$ such that $y=x+ z,$ and so $x+
z\sim_{E_{0}}y.$ By Theorem \ref{th:linear}, $E/\sim_{E_{0}}$ is a
linear PEA.

For any $a\in E_{1},$ we have $ma$ exists and $ma\in E_{m}$ for
$m\in \{1,\ldots,n-1\}$ by Definition \ref{de:npefect}. Now,
$(n-1)a+ ((n-1)a)^{\sim}=1,$ and $((n-1)a)^{\sim}\in E_{1}$.
However, since  $E_{1}$ is downwards directed, and so, there exists
an element $c\in E_{1}$ such that $c\leqslant a,((n-1)a)^{\sim}.$ Whence, for
$i\in \{0,1,\ldots,n\},$ $ic$ exists and $ic\in E_{i}.$ Thus,
$E_{i}= (ic)/\sim_{E_{0}},$ for $i\in \{0,1,\ldots,n\}.$ Hence, we
can define the mapping $\phi:E/\sim_{E_{0}}\rightarrow
\{0,\frac{1}{n},\ldots,1\}$  by $\phi(E_{i})=\frac{i}{n}$ for any
$i\in \{0,1,\ldots,n\},$ which is an isomorphism between effect
algebras.
\end{proof}

\section{Representation of strong $n$-perfect PEA}

In \cite{li08}, the author studied the structure of  non-Archimedean effect algebras and gave some conditions such that a non-Archimedean  effect algebra $E$ is isomorphic to the lexicographical product of one Archimedean effect algebra with a linearly ordered group. We recall that a PEA $E$ is {\it Archimedean} if $\mbox{Infinit}(E)=\{0\}.$

In this section, we introduce a stronger class of $n$-perfect PEAs, called strong $n$-perfect PEAs. We  will give  conditions such that any strong $n$-perfect PEA is isomorphic with the $n$-perfect PEA $\Gamma(\mathbb Z\lex G,(n,0)),$ where $G$ is a torsion-free po-group such that $\mathbb Z\lex G$ satisfies (RDP)$_1.$ In addition, we will study a categorical equivalence of the category of strong $n$-perfect PEAs with a special category of torsion-free directed po-groups.

Let $(E;+,0,1)$ be a PEA and $(G;+,\leqslant)$ be a
directed  po-group with a fixed element $h \in G.$ Let $E\overrightarrow{\times}_h  G$
be the set $\{(0,g)\mid g\in G^{+}\}\cup\{(a,g)\mid a\in E\setminus \{0,1\}, g\in G\}
\cup\{(1,g)\mid g\leqslant h, g\in G\}$, and define a
partial addition $+^{*}$ on $E\overrightarrow{\times}_h G$
componentwise just as following, for any $(a,x), (b,y), (c,z)\in
E\overrightarrow{\times}_h G,$ $(a,x)+^{*}(b,y)$ exists and
equals to $(c,z)$ if and only if $a+ b=c,x+y=z.$ It is routine
to prove that $(E\overrightarrow{\times}_h G; +^{*},(0,0),(1,h))$
is a pseudo-effect algebra. The set $(E\overrightarrow{\times}_h G;
+^{*},(0,0),(1,h))$ is called the
{\it lexicographical product} of the pseudo-effect algebra
$E$ with the po-group $G$ and with respect to the  element $h$ of $G$.
We recall that $E\overrightarrow{\times}_h G$ can be also expressed via the $\Gamma$ functor as $\Gamma(E\overrightarrow{\times} G, (1,h)):=\{(a,g)\mid (0,0)\leqslant (a,g)\leqslant (1,h)\}.$

It is routine to verify the following proposition.

\begin{proposition}\label{pr:lex}
Let $(E;+,0,1)$ be the  effect algebra
$\{0,\frac{1}{n},\ldots,1\}$ and $(G;+,\leqslant)$ be a directed
po-group. Then the lexicographical product
$(E\overrightarrow{\times}_h  G; +^{*},(0,0),(1,h))$  of the effect
algebra $E$ and the po-group $G$ with respect to the element $h$ is
the $n$-perfect PEA $\Gamma(\frac{1}{n}\mathbb Z \overrightarrow{\times} G, ), (1,h)).$
\end{proposition}

\begin{proposition}\label{pr:pogid}
Let  $(E;+,0,1)$ be an $n$-perfect PEA. Then there exists a
unique directed  po-group $G$ such that
$G^{+}=E_{0}$.
\end{proposition}

\begin{proof}
By Theorem \ref{th:infinit}, $E_{0}=\mbox{\rm Infinit}(E)$. Furthermore,
$E_{0}+ E_{0}$ exists and $E_{0}+ E_{0}=E_{0}.$ Hence, for
any $x,y\in E_{0}$, $x+y\in E_0$, and $(E_{0};+,0)$ is a semigroup.
For any $x,y\in E_{0},$ the equation $x+y=0,$ implies that $x=y=0.$
For any $x,y,z\in E_{0},$ the equation $x+y=x+z$ implies that $y=z$,
and equation $y+x=z+x$ implies that $y=z$. Then $(E_{0};+,0)$  is a
cancellative semigroup satisfying the conditions of Birkhoff, \cite[Thm II.4]{Fuc}, which guarantees that $E_{0}$ is the positive
cone of a unique (up to isomorphism) po-group $G$. Without loss of generality, we can assume that $G$ is generated by the positive cone $E_0,$ so that $G$ is directed, see \cite[Prop II.5]{Fuc}.
\end{proof}

\begin{proposition}\label{pr:eq}
Let $(E;+,0,1)$ be an $n$-perfect PEA  and $H$ be a partially
ordered group with strong unit $u$. Assume that $E=\Gamma(H,u)$.
Then the following statement holds.

$(\ast)$ For $x, y\in E\backslash E_{0}$, $a,b,c,d,e,f,g,h\in
E_{0}$, if $x\backslash a=y\backslash b$ and   $x\backslash
c=y\backslash d$, then $b/a=d/c$ and $a/b=c/d$ hold in $G.$ If  $e/x
=f/y$ and $g/x =h/y$,  then  $e\backslash f=g\backslash h$ and
$f\backslash e=h\backslash g$ hold in $H$.
\end{proposition}

\begin{proof} We assume that $(E_{0},\ldots,E_{n})$ is an
$n$-decomposition of $E$ and $G$ is a unique
po-group determined by $E$ such that $G^{+}=E_{0}.$  Since $E$ is an
interval PEA, we assume there exists a positive element $u$ of a
po-group $(H;+,0)$ such that $\Gamma(H,u)=E$. Thus,
$G$ is a subgroup of $H$ with $G^{+}\subseteq H^{+}.$

If $x\backslash a=y\backslash b$ and $x\backslash c=y\backslash d$,
then we have that $x=y+(-b)+a=y+(-d)+c$, $y=x\backslash a+
b=x\backslash c+ d$ and so $-y+x=(-b)+a=(-d)+c,$ $-x+y=-a+b=-c+d,$
which implies $y/x=b/a=d/c,$ $x/y=a/b=c/d$. The proof of the rest  is similar.
\end{proof}

\begin{proposition}\label{le:symmetric}
Let $(E;+,0,1)$ be an $n$-perfect PEA with an $n$-decomposition
$(E_{0},\ldots,E_{n})$. If there exists an element $c\in E_{1}$ such that $nc=1$, then, for any $x\in E,$ $x+ c$ exists if and only if $c+ x$ exists.
\end{proposition}

\begin{proof}
Since $nc=1$,  we have that $c^{\sim}=c^{-}$. Then $x+ c$
exists if and only if $x\leqslant c^{-}$ if and only if
$x\leqslant c^{\sim}$ if and only if $c+ x$ exists.
\end{proof}

The following notions were defined for GMV-algebras in \cite{Dv08}, and cyclic elements were defined also in \cite{225}.

Let $n>0$ be an integer. An element $a$ of a PEA $E$ is said
to be {\it cyclic of order} $n>0$ if $na$ exists in $E$ and $na =
1.$ If $a$ is a cyclic element of order $n$, then $a^- = a^\sim$, indeed, $a^-
= (n-1) a = a^\sim$.

We say that a group $G$ is {\it torsion-free} if $ng\ne 0$ for any $g\ne 0$ and every nonzero integer $n.$  For example, every $\ell$-group is torsion-free, see \cite[Cor 2.1.3]{Gla}. We recall that if $G$ is torsion-free, so is $\mathbb Z \overrightarrow{\times} G.$

We recall that a group $G$ enjoys {\it unique  extraction of
roots} if, for all positive integers $n$ and $g,h \in G$, $g^n =
f^n$ implies $g=h$.  We recall that every linearly ordered group, or
a representable $\ell$-group, in particular every Abelian
$\ell$-group enjoys unique  extraction of roots, see \cite[Lem.
2.1.4]{Gla}.

We say that a PEA $E$ enjoys
{\it unique  extraction of roots of $1$} if $a,b \in E$ and $n a, n
b$ exist in $E$, and $n a=1= n b$, then $a= b.$  Then every $\Gamma(\mathbb Z \lex G,(n,0)) $ enjoys unique  extraction of roots of $1$ for any $n\ge 1$
and any torsion-free directed po-group $G$.  Indeed, let $k(i,g) = (n,0)=k(j,h)$. Then $ki=n=kj$
which yields $i=j >0$, and $kg=0=kh$ implies $g=0=h$.

\begin{definition}\label{de:strongper}
{\rm  Let $E$ be an $n$-perfect PEA satisfying (RDP)$_1$. We say that $E$ is a {\it
strong} $n$-perfect PEA  if

\begin{itemize}
\item[(i)] there exists a torsion-free unital po-group $(H,u)$ such that
$E=\Gamma(H,u),$

\vspace{-2mm}
\item[(ii)] there exists an element $c\in E_{1}$ such that (a) $nc=u,$ and (b) $c\in C(H).$

\end{itemize}
}
\end{definition}

The element $c$ from (ii) is said to be a {\it strong cyclic element of order
$n$.}

\begin{lemma}\label{le:unique}
Any strong cyclic element $c$ of order $n$   is a unique element $d\in E=\Gamma(H,u)$ such that $nd=u$  implies $c=d$ whenever $H$ is torsion-free.
\end{lemma}

\begin{proof}
Indeed, since $c \in C(H)$ and $d\in H,$ we have $c+d=d+c$ in the group $H.$ Then $n(c-d)=nc-nd=0$ so that $c=d.$
\end{proof}

\begin{theorem}\label{th:nperpog}
Let $E$ be a  PEA   and let $n\ge 1$ be an integer. Then $E$ is a strong $n$-perfect PEA if and only if
there exists a torsion-free directed  po-group $G$ such that $\mathbb Z\lex G$ satisfies {\rm(RDP)}$_1,$ and
$E$ is isomorphic to
$\Gamma(\mathbb{Z}\overrightarrow{\times}G,(n,0)).$

If it is a case, $G$ is unique and satisfies {\rm(RDP)}$_1.$
\end{theorem}

\begin{proof}  If there exists a torsion-free directed po-group $G$  such that  $E$ is isomorphic to
$\Gamma(\mathbb{Z}\overrightarrow{\times}G,(n,0)),$ then $E$ is an $n$-perfect PEA with a unique strong cyclic element $(1,0)$ of order $n.$ Hence, $E$ is a strong $n$-perfect PEA.

Conversely, assume  that $E$ is a strong $n$-perfect PEA with an $n$-decomposition $(E_{0},\ldots,E_{n})$ and $E=\Gamma(H,u)$ for a torsion-free unital  po-group $(H,u)$ satisfying (RDP)$_1.$ By \cite[Thm 5.7]{DvVe01b}, $(H,u)$ is a unique (up to isomorphism of unital po-groups) unital po-group with (RDP)$_1.$  By Lemma \ref{le:unique}, there exists a unique strong cyclic element $c\in E_{1}$ such that $nc=u$ and $c+g=g+c$ for any $g\in H$. Thus,
$E_{i}= (ic)/\sim_{E_{0}}$ for $i\in \{0,1,\ldots,n\}.$
Furthermore, by Proposition \ref{pr:pogid}, there exists a directed po-group $G$ such that $G$ is a subgroup of $H$ and
$E_{0}=G^{+}\subseteq H^{+}.$

We define a mapping $\varphi:E\rightarrow \Gamma(\mathbb{Z}\overrightarrow{\times}G,(n,0))$ as follows, if $x\in E_{i},$ then
$\varphi(x)=(i, (ic)/x),$ $i\in\{0,1,\ldots,n\}$.
Since $E$ is an interval PEA, by Proposition \ref{pr:eq},  the  condition  $(\ast)$ holds, which implies that $\phi$ is defined well, and $\varphi(0)=(0,0)$, $\varphi(u)=(n,0).$

Assume  $x+ y$ exists in $E$ for $x,y\in E.$ Then there exist
unique $i,j\in \{0,1,\ldots,n\}$ such that $x\in E_{i},$ $y\in E_{j},$ and so $x+ y\in E_{i+j}$. By the definition of
$\varphi$, we have that  $\varphi(x+ y)=(i+j,((i+j)c)/(x+ y)).$ Since for any $g\in H,$ $c+g=g+c,$ we have that $-c+g=g-c,$
which implies $((i+j)c)/(x+ y)$
$=(ic+jc)/(x+y)=-(ic+jc)+x+y=-jc-ic+x+y=
-ic+x-jc+y=(ic)/x+(jc)/y,$ and
so $(i+j,((i+j)c)/(x+ y)) =(i+j,(ic)/x+(jc)/y)$
$=(i,(ic)/x)+(j,(jc)/y)$, which implies $\varphi(x+ y)=\varphi(x)+\varphi(y)$. Thus, $\varphi$ is a morphism
between pseudo-effect algebras. Assume $\varphi(x)=(i,(ic)/x),$
$\varphi(y)=(j,(jc)/y),$ and $\varphi(x)\leqslant \varphi(y).$ There are the following two  cases. (i) If $i=j,$ then $(ic)/x\leqslant (jc)/y,$
and so $x\leqslant y.$ (ii) If $i<j$,  then $x\in E_{i},$ $y\in E_{j},$ which implies that $x<y$ by Theorem \ref{th:infinit} (ii).
Thus, $\varphi$ is a monomorphism. For any $g\in G^{+}$, $\varphi^{-1}((0,g))=g\in E_{0},$ $\varphi^{-1}((n,-g))=g^{-}\in E_{n}$. Assume that $(i, g)\in\Gamma(\mathbb{Z}\overrightarrow{\times}G,(n,0))$ and $i\in\{1,\ldots,n-1\}.$  Then $ic+ g=\varphi^{-1}((i,g))$. Thus,
$\varphi$ is surjective.

Hence, $\varphi$ is an isomorphism from the  strong $n$-perfect  pseudo-effect algebra $E$ onto the PEA $\Gamma(\mathbb{Z}\overrightarrow{\times}G,(n,0)).$ Since $(H,u)$ is a unique unital po-group with (RDP)$_1$ such that $E=\Gamma(H,u)$ and $E$ is isomorphic with $\Gamma(\mathbb Z\lex G,(n,0)),$ we have by \cite[Thm 5.7]{DvVe01b} that $(H,u)$ and $(\mathbb Z\lex G,(n,0))$ are isomorphic unital po-groups with (RDP)$_1$ and whence, $\mathbb Z\lex G$ is torsion-free. We show that also  $G$ is torsion-free. Indeed, assume that, for some integer $n\ne 0$ and some element $g \in G,$ we have $ng=0.$ But $ng$ belongs also to the torsion-free po-group $H$, whence, $g=0.$

Finally, since $\mathbb Z\lex G$ satisfies (RDP)$_1$, then clearly so does $G.$
\end{proof}

It is worthy to recall that we do not know whether if a directed po-group $G$ has (RDP)$_1$, does have (RDP)$_1$ also $\mathbb Z\lex G$ ? This is know only for Abelian po-group, see  \cite[Cor 2.12]{Good86}.

\begin{corollary}\label{co:uniqc}
Let  $E$ be a strong $n$-perfect PEA. Then:
\begin{itemize}
\item[{\rm(i)}] There exists a unique strong cyclic element of order $n$ in $E$.

\vspace{-2mm}\item[{\rm(ii)}] There exists a  unique $n$-decomposition $(E_{0},\ldots,E_{n})$
of $E$.

\vspace{-2mm}\item[{\rm(iii)}]  The state $s:E\rightarrow [0,1]$ such that
$s(E_{i})=\frac{i}{n}$ for $i\in\{0,\ldots,n\}$ is extremal.

\vspace{-2mm}\item[{\rm(iv)}] $E$ is a symmetric PEA.
\end{itemize}

\end{corollary}

\begin{proof}
By Theorem \ref{th:nperpog}, there exists a torsion-free directed po-group
$G$ such that $E$ is isomorphic to
$\Gamma(\mathbb{Z}\overrightarrow{\times}G,(n,0)).$

(i) The element $(1,0)$ is a unique strong cyclic element of order $n$
of $\Gamma(\mathbb{Z}\overrightarrow{\times}G,(n,0)),$ which implies
that there exists a unique strong cyclic element of order $n$ in
$E,$ see see Lemma \ref{le:unique}.

(ii) The pseudo-effect algebra
$\Gamma(\mathbb{Z}\overrightarrow{\times}G,(n,0))$ admits a unique
$n$-decomposition $(E_{0},\ldots,E_{n}),$ where $E_{0}=\{(0,g)\mid
g\in G^{+}\},$ $E_{n}=\{(n,-g)\mid g\in G^{+}\},$ $E_{i}=\{(i,g)\mid
g\in G\},$ for $i\in \{0,\ldots,n-1\}.$

(iii) It easy to see that a function
$s:\Gamma(\mathbb{Z}\overrightarrow{\times}G,(n,0))\rightarrow
[0,1]$ such that $s(i,g)=\frac{i}{n}$ for $i\in\{0,\ldots,n\}$ is a unique state on
$\Gamma(\mathbb{Z}\overrightarrow{\times}G,(n,0))$. Indeed, let  $s_1$ be a state on $\Gamma(\mathbb{Z}\overrightarrow{\times}G,(n,0)).$ It is clear that $E_0=\mbox{Infinit}(E) \subseteq \Ker(s_1).$  On the other hand $\Ker(s_1)\subseteq E_0$ because, $E_0$ is a maximal ideal, which yields $\Ker(s_1)=E_0.$ Moreover, $1=s_1(n(1,0))= ns_1(1,0)$ which gives $s_1(1,0) = 1/n.$ Let $g \geqslant 0,$ then $(1,g)=(1,0)+(0,g)$ which yields $s_1(1,g)=1/n.$  Let $g\in G$ be arbitrary. Since $G$ is  directed, every element $g= g_1-g_2$ for some $g_1,g_2\geqslant 0.$  Hence, $(1,g) \leqslant (1,g_1)$ which entails $s_1(1,g) \leqslant \frac{1}{n},$ and similarly, $s_1(i,g)\leqslant \frac{i}{n}$ for $i=1,\ldots,n-1.$  Therefore, $1=s_1(1,g)+s_1(n-1,-g)\leqslant \frac{1}{n} + \frac{n-1}{n}=1$ which implies $s_1(1,g)=\frac{1}{n}$ and $s_1(E_i) = \frac{i}{n}$ for any $i=0,1,\ldots,n.$

Hence, $E$ admits a
unique state $s$ such that $s(E_{i})=\frac{i}{n}$ for
$i\in\{0,\ldots,n\}$, and so it is also extremal.

(iv) The PEA
$\Gamma(\mathbb{Z}\overrightarrow{\times}G,(n,0))$ is a symmetric
PEA, and so $E$ is also symmetric.
\end{proof}

\begin{theorem}\label{th:hom}
Let  $E$ and $F$ be two strong $n$-perfect PEAs, and
$(E_{0},\ldots,E_{n})$ and $(F_{0},\ldots,F_{n})$ be
$n$-decompositions of $E$ and $F$, respectively. If $f:E\rightarrow
F$ is a homomorphism between  $E$ and $F$ and if $G$ and $H$ are
two directed  po-groups which are determined by $E$
and $F,$ respectively, by the property $G^{+}=E_{0}$ and $H^{+}=F_{0}$, then
\begin{itemize}
\item[{\rm(i)}] $f(E_{i})\subseteq F_{i}$, for $i\in \{0,1,\ldots,n\},$

\item[{\rm(ii)}] there exists a unique homomorphism $\widehat{f}:G\rightarrow H$
such that for any $g\in G^{+}$, $\widehat{f}(g)=h$ iff
$f(g)=h.$
\end{itemize}
\end{theorem}

\begin{proof} Since  $E$ and $F$ are two strong $n$-perfect PEAs, by Theorem \ref{th:nperpog}, there exist unique directed torsion-free  po-groups  $G$ and $H$ with (RDP)$_1$
such that $E$ and $F$ are isomorphic to
$\Gamma(\mathbb{Z}\overrightarrow{\times}G,(n,0))$ and
$\Gamma(\mathbb{Z}\overrightarrow{\times}H,(n,0))$. Thus, we can
assume that $E=\Gamma(\mathbb{Z}\overrightarrow{\times}G,(n,0))$ and
$F=\Gamma(\mathbb{Z}\overrightarrow{\times}H,(n,0))$, and
$E_{0}=\{(0,g)\mid g\in G^{+}\}$, $E_{n}=\{(0,-g)\mid g\in G^{+}\}$,
$F_{0}=\{(0,h)\mid h\in H^{+}\},$ $F_{n}=\{(0,-h)\mid h\in H^{+}\}$,
and for any $i\in\{1,\ldots,n-1\},$ $E_{i}=\{(i,g)\mid g\in G\}$,
$F_{i}=\{(i,h)\mid h\in H\}.$

(i) For any $(0,g)\in E_{0}=\mbox{\rm Infinit}(E)$ and any integer $k\ge 1$, $k(0,g)=(0,kg)\in E_0,$ we have $f(0,kg)=kf(0,g)$, we conclude $f(0,g)\in F_0=\mbox{\rm Infinit}(F).$  Thus, we have that $f(E_{0})\subseteq
F_{0}.$ Further, by $E_{n}=E_{0}^{-}$  and $F_{n}=F_{0}^{-}$, we
have that $f(E_{n})\subseteq F_{n}.$

Assume $n>1.$
For  $(1,0)\in E_{1},$  by $n(1,0)=(n,0),$ we have that $nf(1,0)=(n,0),$ which implies that $f(1,0)\in F_{1}.$ For any
$(1,g),$ there exists $g_{1},g_{2}\in G^{+}$ such that $g=g_{1}-g_{2}$ and so $(1,g)=(0,g_{1})+(1,-g_{2}).$ Now,
$f(1,-g_{2})\leqslant f(1,0)$ which entails $f(1,-g_{2})\in F_0\cup F_{1}.$
But $ f(1,0)= f(1,-g_2)+ f(0,g_2)\in F_1$ and $f(0,g_2) \in F_0,$ so that $f(1,-g_2)\in F_1.$   Consequently, $f(1,g) = f(0,g_1)+f(1,-g_2) \in F_1$ for any $g \in G.$

In the same way, we can show that $f(i,g)\in F_j$ for some $j=0,1,\ldots,i<n.$  In any rate, we state $f(1,g) \in F_1.$ If not, then $f(1,g)\in F_0$ and $f(n-1,g)\in F_j$ for some $j=0,1,\ldots,i.$ But $f(n,0)\in F_n$ and $f(n,0)=f(1,g)+f(n-1,-g) \in F_0+F_j = F_j \subseteq E\setminus F_n,$ which is absurd.

Since for any $i\in \{1,\ldots,n-1\},$ $E_{i}=iE_{1}$, and so we have that $f(E_{i})\subseteq F_{i}$.

(ii) We define $f_{1}:G^{+}\rightarrow H^{+}$ as follows:  for $g\in  G^{+}$,   $f_{1}(g)=h$ iff   $f(0,g)=(0,h).$ Obviously, $f_{1}$
is defined well and $f(0,g)=(0,f_{1}(g))$ for $g\in G^{+}$.
Furthermore, for any $g_{1},g_{2}\in G^{+},$ by
$f(0,g_{1}+g_{2})=f(0,g_{1})+ f(0,g_{2}) $ and $f(0,g)=(0,f_{1}(g))$, we have that
$f_{1}(g_{1}+g_{2})=f_{1}(g_{1})+f_{1}(g_{2}).$

We now define $\widehat{f}:G\rightarrow H$ as follows: for any
$g\in G,$ we $\widehat{f}(g)=f_1(g_{1})-f_1(g_{2})$ whenever $g=g_1-g_2,$ where $g_1,g_2 \geqslant 0.$  We assert that $\widehat f$ is a well-defined mapping.  Indeed, if $g=-h_1+h_2$ for some $h_1,h_2 \geqslant 0,$ then
$g=g_{1}-g_{2}=-h_{1}+h_{2},$
and  $h_{1}+g_{1}=h_{2}+g_{2},$ which
implies that $f_1(h_{1}+g_{1})=f_1(h_{2}+g_{2}),$ and so $f_1(g_{1})-f_1(g_{2})=-f_1(h_{1})+f_1(h_{2}).$ Thus, $\widehat{f}$ is
defined well.

For $g,h\in G,$ we want to verify that $\widehat{f}(g+h)=\widehat{f}(g)+\widehat{f}(h).$ We assume that $g=-g_{1}+g_{2},$ $h=h_{1}-h_{2},$ and $g+h=k_{1}-k_{2},$ then
$g_{2}+h_{1}-h_{2}=g_{1}+k_{1}-k_{2}$, which entails that
$g_{2}+h_{1}-h_{2}=g_{1}+k_{1}-k_{2},$ and so,
$f_1(g_{2})+f_1(h_{1})-f_1(h_{2})=f_1(g_{1})+f_1(k_{1})-f_1(k_{2}),$ hence,
$\widehat{f}(g+h)=\widehat{f}(g)+\widehat{f}(h).$ Thus,
$\widehat{f}$ is a group homomorphism.  Moreover, if $g\geqslant 0,$ then $\widehat f(g)\geqslant 0.$ Furthermore, by the
definition of $\widehat{f},$ we have that
$\widehat{f}(G^{+})\subseteq H^{+},$ and    $g\in G^{+}$, $\widehat{f}(g)=h$.

Now, if $k:G\rightarrow H$ is a homomorphism such that for $g\in G^{+}$,
$\widehat{f}(g)=h$ iff $f(0,g)=(0,h)$. Then
$\widehat{f}\mid_{G^{+}}=k_{G^{+}},$ since $G$ is directed, we have
that  $\widehat{f}=k.$
\end{proof}

Let $\mathcal{G}$ be the category whose objects are  torsion-free directed  po-groups $G$ such that $\mathbb Z\lex G$ satisfies (RDP)$_1$  and morphisms are  po-group homomorphisms. Let $\mathcal{SPPEA}_{n}$ be the category whose objects are strong
$n$-perfect PEAs and morphisms are homomorphisms of PEAs.

We define a functor
$\mathcal{E}_n:\mathcal{G}\rightarrow\mathcal{SPPEA}_{n}$ as follows: for $G\in \mathcal{G},$ let
$\mathcal{E}_n(G):=\Gamma(\mathbb{Z}\overrightarrow{\times}G,(n,0))$
and if $h$ is a group homomorphism with domain $G,$ we set
$$\mathcal{E}_n(h)(x)=(i,h((ic)/x)),$$
where $c$ is a unique strong cyclic
element of order $n$ in $E.$

\begin{theorem}\label{th:faifull}
The functor
$\mathcal{E}_n$ is a faithful and full functor from the category $\mathcal{G}$ of directed po-groups into the category
$\mathcal{SPPEA}_{n}$ of n-perfect PEAs.
\end{theorem}

\begin{proof}
Let $h_{1}$ and $h_{2}$ be two morphisms from $G_{1}$ into $G_{2}$ such that $\mathcal{E}(h_{1})=\mathcal{E}(h_{2}).$ Since both $G_{1}$ and $G_{2}$ are directed, it suffice to prove that $h_{1}(g)=h_{2}(g)$ for all $g\in G_{1}^{+}.$ By $\mathcal{E}(h_{1})=\mathcal{E}(h_{2}),$ then
$(0,h_{1}(g))=(0,h_{2}(g))$ for all $g\in G_{1}^{+},$ and hence $h_{1}=h_{2}.$

Let $f:\Gamma(\mathbb{Z}\overrightarrow{\times}G_{1},
(n,0))\rightarrow \Gamma(\mathbb{Z}\overrightarrow{\times}G_{2},(n,0))$ be a PEA homomorphism. Then for any $x\in G_{1}^{+}$,
there exists a unique $y\in G_{2}^{+}$ such that $f(0,x)=(0,y).$
Define a mapping $h:G_{1}^{+}\rightarrow G_{2}^{+}$ by $h(x)=y$
iff $f(0,x)=(0,y).$ Note for any $x_{1}, x_{2}\in G_{1}^{+},$
$h(x_{1}+ x_{2})=h(x_{1})+h(x_{2}).$ Since $G_{1}$ is directed, for
any $x\in G_{1}$, there exists $x_{1}, x_{2}, g_{1}, g_{2}\in
G_{1}^{+}$ such that $x=x_{1}-x_{2}$ and $x=-g_{1}+g_{2},$ then
$g_{1}+x_{1}=g_{2}+x_{2}$, and so
$h(g_{1})+h(x_{1})=h(g_{2})+h(x_{2}),$ which implies that $h(x_{1})-h(x_{2})=-h(g_{1})+h(g_{2})$. This shows that the assignment $h(x)=h(x_{1})-h(x_{2})$ is a well-defined extension of $h$ to the whole directed  po-group $G_{1},$ and $h$ is a po-group homomorphism.
\end{proof}

We say that  a {\it universal} group for a PEA $E$ is a pair $(G,\gamma)$ consisting of a directed po-group $G$ and a $G$-valued measure
$\gamma:E\rightarrow G$ (i.e., $\gamma(a+ b)=\gamma(a)+
\gamma(b)$ whenever $a+ b$ exists in $E$) such that  the following conditions hold:
(i) $\gamma(E)$ generates $G,$ and (ii) if $H$ is a group and
$\phi: E\rightarrow H$ is an $H$-valued
measure, then there exists a (unique) group homomorphism
$\phi^{\ast}:G\rightarrow H$ such that
$\phi=\phi^{\ast}\circ\gamma.$

\begin{theorem}\label{th:uni}
Let $E$ be a strong n-perfect PEA. Then the directed  po-group  $\mathbb{Z}\overrightarrow{\times}G$ from Theorem {\rm \ref{th:nperpog}} together with the isomorphism $\gamma: E\to \Gamma(\mathbb Z\lex G,(n,0))\subset \mathbb Z\lex G$ is  a universal group of $E$.
\end{theorem}

\begin{proof}
Let $E$ be a strong $n$-perfect PEA. By Theorem \ref{th:nperpog}, there is a unique torsion-free directed po-group $G$ such that $E$ is isomorphic with $\Gamma(\mathbb Z\lex G,(n,0)).$
Set $\mathcal{G}=\mathbb{Z}\overrightarrow{\times}G,$ and
$\gamma:E\rightarrow\mathbb{Z}\overrightarrow{\times}G$ be the embedding mapping, then:

(i) $\gamma(E)$ generates the group $\mathcal{G}$ and because $(1,0)$ is
a strong unit, $\mathcal{G}$ is directed.

(ii) Assume $\phi: E\rightarrow K$ is a $K$-valued measure. Then $\phi(0,0)=0_{H}.$ Notice that $E$ is symmetric, and so for any $g\in G^{+},$
$\phi(1,-g)=\phi((1,0)\backslash(0,g))=\phi((0,g)/(1,0)),$ which implies that $\phi(1,0)-\phi((0,g)=-\phi((0,g)+\phi(1,0).$ Define a mapping $\phi^{\ast}:\mathcal{G}\rightarrow H,$ as follows, for any $g,h\in G^{+}$,

(a) $\phi^{\ast}(0,g)=\phi(0,g)$,

(b) $\phi^{\ast}(0,-g)=-\phi(0,g)$,

(c) $\phi^{\ast}(0,g-h)=\phi(0,g)-\phi(0,h),$

(d) $\phi^{\ast}(0,-g+h)=-\phi(0,g)+\phi(0,h),$

(e) $\phi^{\ast}(1,g-h)=\phi(1,0)+\phi^{\ast}(0,g-h),$

(f) $\phi^{\ast}(m,g-h)=m\phi(1,0)+\phi^{\ast}(0,g-h).$

For $g\in G$, if there exist $g_{1},g_{2},h_{1},h_{2}\in G^{+},$
such that $(0,g)=(0,g_{1}-g_{2})=(0,h_{1}-h_{2}),$ then there exist $k_{1},k_{2}\in G^{+},$ such that $g=-k_{1}+k_{2},$ since $G$ is a directed po-group. Thus, we have that
$k_{1}+g_{1}=k_{2}+g_{2},$ $k_{1}+h_{1}=k_{2}+h_{2},$ which implies
that $\phi(0,k_{1})+\phi(0,g_{1})=\phi(0,k_{2})+\phi(0,g_{2}),$
$\phi(0,k_{1})+\phi(0,h_{1})=\phi(0,k_{2})+\phi(0,h_{2}),$ thus,
$-\phi(0,k_{1})+\phi(0,k_{2})=\phi(0,g_{1})-\phi(0,g_{2})=\phi(0,h_{1})-
\phi(0,h_{2}).$
Consequently, $\phi^{\ast}$ is defined well.

For any $g_{1},g_{2}, h_{1},h_{2}\in G^{+},$  there exists $k_{1},k_{2}\in G^{+}$ such that $ -g_{2}+ h_{1}=k_{1}-k_{2},$ and
so  $(0,g_{1}-g_{2})+(0,h_{1}-h_{2})=(0,g_{1}+k_{1}-k_{2}-h_{2}).$
Hence, $\phi^{\ast}((0,g_{1}-g_{2})+(0,h_{1}-h_{2}))$
$=\phi^{\ast}(0,g_{1}+k_{1}-k_{2}-h_{2})$
$=\phi(0,g_{1}+k_{1})-\phi(0,h_{2}+k_{2})$
$=\phi(0,g_{1})+\phi(0,k_{1})-\phi(0,k_{2})-\phi(0,h_{2})$
$=\phi(0,g_{1}+k_{1})-\phi(0,h_{2}+k_{2})$
$=\phi(0,g_{1})+\phi(0,k_{1})-\phi(0,k_{2})-\phi(0,h_{2})$
$=\phi(0,g_{1})+\phi^{\ast}(0,k_{1}-k_{2})-\phi(0,h_{2})$
$=\phi(0,g_{1})+\phi^{\ast}(0,-g_{2}+h_{1})-\phi(0,h_{2})$
$=\phi(0,g_{1})-\phi(0,g_{2})+\phi(0,h_{1})-\phi(0,h_{2})$
$=\phi(0,g_{1})-\phi(0,g_{2})+\phi(0,h_{1})-\phi(0,h_{2})$
$=\phi^{\ast}(0,g_{1}-g_{2})+\phi^{\ast}(0,h_{1}-h_{2}).$

Since for any $g\in G^{+}$, we have that
$\phi(1,0)-\phi(0,g)=-\phi(0,g)+\phi(1,0),$ which implies that
$\phi(0,g)+\phi(1,0)-\phi(0,g)=\phi(1,0),$ and so $\phi(0,g)+\phi(1,0)=\phi(1,0)+\phi(0,g).$ Thus, for any $g\in G,$
we have that
$\phi(1,0)+\phi^{\ast}(0,g)=\phi^{\ast}(0,g)+\phi(1,0).$

For $g_{1},g_{2}, h_{1},h_{2}\in G^{+},$
$(i,g_{1}-g_{2})+(j,h_{1}-h_{2})=(i+j,g_{1}-g_{2}+h_{1}-h_{2}),$ and so, $\phi^{\ast}((i,g_{1}-g_{2})+(j,h_{1}-h_{2}))$
$=\phi^{\ast}(i+j,g_{1}-g_{2}+h_{1}-h_{2})$
$=(i+j)\phi(1,0)+\phi^{\ast}(0,g_{1}-g_{2}+h_{1}-h_{2})$
$=i\phi(1,0)+j\phi(1,0)+\phi^{\ast}(0,g_{1}-g_{2})+\phi^{\ast}(0,h_{1}-h_{2})$
$=i\phi(1,0)+\phi^{\ast}(0,g_{1}-g_{2})+j\phi(1,0)+\phi^{\ast}(0,h_{1}-h_{2})$
$=\phi^{\ast}(i,g_{1}-g_{2})+\phi^{\ast}(j,h_{1}-h_{2}).$

Thus, $\phi^{\ast}$ is a group homomorphism with $\phi=\phi^{\ast}\circ\gamma.$
\end{proof}

Recall that a functor $\mathcal{E}$ from  a category $ \mathcal A$ into a category
$\mathcal B$ is said to be {\it left-adjoint} provided that for every $\mathcal B$-object
$B$ there exists an $\mathcal{E}$-universal arrow with domain $B,$  see \cite{AHG00}.

\begin{theorem}\label{th:weper}
The functor $\mathcal{E}_n$ has a left-adjoint.
\end{theorem}

\begin{proof} By Theorem \ref{th:nperpog}, for any  strong $n$-perfect PEA $E,$  there
exists a unique torsion-free directed  po-group $G$ such that $\mathbb Z \lex G$ has (RDP)$_1,$   and by Theorem \ref{th:uni}, $(\mathbb{Z}\overrightarrow{\times}G,\gamma)$  is a universal group for $E$.

Now, for a strong $n$-perfect PEA $F=\mathcal E_n(G_1),$ where $G_1$ is a torsion-free directed po-group such that $\mathbb Z\lex G$ has (RDP)$_1,$ assume that $f^{\prime}:E\rightarrow \mathcal{E}_n({G_{1}})$ is a homomorphism between PEAs.  There exists a unique group homomorphism
$f_{1}:\mathbb{Z}\overrightarrow{\times}G\rightarrow\mathbb{Z}\overrightarrow{\times}G_{1}$
such that $f^{\prime}=f_{1}\circ\gamma$. Now, we define $f:G\rightarrow G_{1}$ as $f(g)=h$ iff $f_{1}(0,g)=(0,h)$ for any
$g\in G^{+},$ since $G$ is directed, it is routine to verify that
$f$ is a   homomorphism between $G$ and $G_{1}.$ By Theorem \ref{th:hom}, $f$ is a unique homomorphism between $G$ and $G_{1} $
such that $f(g)=h$ iff $f_{1}(0,g)=(0,h)$ for any $g\in G^{+},$
which implies that it is  also a unique homomorphism such that
$f^{\prime}=\mathcal{E}(f)\circ\gamma.$
\end{proof}

We define a functor
$\mathcal{P}_n:\mathcal{SPPEA}_{n}\rightarrow\mathcal{G}$ as follows: if $G\in \mathcal{G},$ let
$$
\mathcal{P}_n(\Gamma(\mathbb{Z}\overrightarrow{\times}G,(n,0))):=G,
$$
where $(\mathbb{Z}\overrightarrow{\times}G,\gamma)$ is a universal group of the strong $n$-perfect PEA
$\Gamma(\mathbb{Z}\overrightarrow{\times}G,(n,0)).$

\begin{theorem}\label{th:weper1}
The functor $\mathcal{P}_n$ is a left-adjoint of the functor $\mathcal{E}_n$.
\end{theorem}

\begin{proof}
It follows from Theorem \ref{th:weper} and the definition of  $\mathcal{P}_n.$
\end{proof}

Recall that a functor $\mathcal F$ from a category $\mathcal A$ into a category $\mathcal B$ is called a {\it categorical  equivalence} provided that it is full, faithful, and
isomorphism-dense in the sense that for any $\mathcal B$-object $B$ there exists some $\mathcal A$-object $A$ such that $\mathcal F(A)$ is isomorphic to $B,$ see \cite{AHG00}.

\begin{theorem}\label{th:weper2}
The functor $\mathcal{E}_n$ is a categorical equivalence of the category $\mathcal{G}$ of directed torsion-free  po-groups $G$ such that $\mathbb Z \lex G$ has {\rm (RDP)}$_1$ and the category
$\mathcal{SPPEA}_{n}$  of $n$-strong perfect pseudo-effect algebras.
\end{theorem}

\begin{proof} It suffices to prove that for a strong $n$-perfect pseudo-effect algebra $E$, there is a torsion-free directed po-group  $G$ such that $\mathbb Z\lex G$ has (RDP)$_1$ and
such that $\mathcal{E}_n(G)$ is isomorphic to $E$. To show that, we take a universal group
$(\mathbb{Z}\overrightarrow{\times}G,\gamma)$. Then $\mathcal{E}_n(G)$ and $E$ are isomorphic.
\end{proof}

{\bf Acknowledgement:} The authors thank  for the support by SAIA,
n.o. (Slovak Academic Information Agency) and the Ministry of
Education, Science, Research and Sport of the Slovak Republic. This
work is also supported by National Science Foundation of China
(Grant No. 60873119),  and the Fundamental Research Funds for the
Central Universities (Grant No. GK200902047).

A.D. thanks  for the support by Center of Excellence SAS -~Quantum
Technologies~-,  ERDF OP R\&D Project
meta-QUTE ITMS 26240120022, the grant VEGA No. 2/0059/12 SAV and by
CZ.1.07/2.3.00/20.0051  and  MSM 6198959214.


\begin{thebibliography}{99}

\bibitem{AHG00} J. Ad\'{a}mek, H. Herrlich, G. E. Strecker,
{\it ``Abstract and Concrete Categories: The Joy of Facts",} Originally published by:
John Wiley and Sons, New York, 1990.
Republished in:
Reprints in Theory and Applications of Categories, No. 17 (2006) pp. 1--507.
http://www.tac.mta.ca/tac/reprints/articles/17/tr17.pdf



\vspace{-2mm}\bibitem{BDL93} L.P. Belluce, A. Di Nola,  A. Letieri,
Local MV-algebras,   {\it Rendiconti del Circolo Matematico di
Palermo} {\bf 42} (1993), 347--361.

\vspace{-2mm}\bibitem{BeFo97}   M. K. Bennett and D. J. Foulis,
Interval and scale effect algebras, {\it Advances in Applied Mathematics} {\bf19} (1997), 200--215.

\vspace{-2mm}\bibitem{209} D. Buhagiar,  E. Chetcuti,  A. Dvure\v censkij,
Loomis-Sikorski representation of  monotone $\sigma$-complete effect
algebras, {\it Fuzzy Sets and Systems} {\bf 157} (2006), 683--690.


\vspace{-2mm}\bibitem{DDJ} A. Di Nola,  A. Dvure\v{c}enskij, J. Jakub\'{i}k,
Good and bad inifinitesimals and states on pseudo MV-effect
algebras, {\it Order} {\bf 21} (2004), 293--314.

\vspace{-2mm}\bibitem{DDT} A. Di Nola,  A. Dvure\v{c}enskij, C. Tsinakis, On perfect GMV-algebras,
{\it Communications in Algebra} {\bf 36} (2008), 1221--1249.

\vspace{-2mm}\bibitem{DL94} A. Di Nola,  A. Lettieri, Perfect  MV-algebras are categorical
equivalent to Abelian $\ell$-groups, {\it Studia Logica} {\bf 53} (1994),
417--432 .

\vspace{-2mm}\bibitem{93}  A. Dvure\v censkij,    {\it  ``Gleason's Theorem and Its
Applications"}, Kluwer Academic Publisher,
Dordrecht/Boston/London, 1993, 325+xv pp.

\vspace{-2mm}\bibitem{Dv02}   A. Dvure\v{c}enskij,
Pseudo MV-algebras are intervals in $\ell$-group, {\it Journal of the
Australian Mathematical Society} {\bf 72} (2002), 427--445.


\vspace{-2mm}\bibitem{Dv03}   A. Dvure\v{c}enskij,
Ideals of pseudo-effect algebras and their applications, {\it Tatra
Mountains Mathematical Publications} {\bf27} (2003), 45--65.

\vspace{-2mm}\bibitem{Dv04}   A. Dvure\v{c}enskij,  States and radicals of pseudo-effect algebras,
{\it Atti del Seminario Matematico e Fisico dell'Universit\`{a} di
Modena} {\bf 52} (2004), 85--103.

\vspace{-2mm}\bibitem{Dv05} A. Di Nola,  A. Dvure\v censkij, M. Hy\v cko, C. Manara,
Entropy on  effect algebras with the Riesz decomposition property II: MV-algebras, {\it  Kybernetika} {\bf 41}  (2005), 161--176.

\vspace{-2mm}\bibitem{Dv07} A. Dvure\v{c}enskij,
Perfect effect algebras are categorically equivalent with Abelian
interpolation po-groups, {\it  Journal of the Australian
Mathematical Society} {\bf 82}  (2007), 183-207.

\vspace{-2mm}\bibitem{Dv08}   A. Dvure\v{c}enskij,  On $n$-perfect GMV-algebras,
{\it Journal of Algebra} {\bf 319} (2008), 4921--4946.

\vspace{-2mm}\bibitem{225} A. Dvure\v censkij, Cyclic elements and subalgebras of
GMV-algebras,  {\it Soft Computing} {\bf 14} (2010), 257--264.

\vspace{-2mm}\bibitem{DvPu00}  A. Dvure\v{c}enskij and  S. Pulmannov\'{a},
{\it ``New Trends in Quantum Structures",} Kluwer Academic  Publishers/Ister
Science, Dordrecht/Bratislava, 2000.

\vspace{-2mm}\bibitem{DvVe01a} A. Dvure\v{c}enskij and T. Vetterlein, Pseudoeffect
algebras. I. Basic properties, {\it International Journal of
Theoretical Physics} {\bf 40} (2001), 685--701.

\vspace{-2mm}\bibitem{DvVe01b} A. Dvure\v censkij, T. Vetterlein,
Pseudoeffect algebras. II. Group representation, {\it International Journal of Theoretical Physics} {\bf 40} (2001), 703--726.

\vspace{-2mm}\bibitem{DvVe02}   A. Dvure\v{c}enskij and T. Vetterlein,
Algebras in the positive cone of po-groups, {\it Order} {\bf 19} (2002),
127--146.

\vspace{-2mm}\bibitem{FoBe94}   D. J. Foulis and M. K. Bennett,
Effect algebras and unsharp quantum logics, {\it Foundations of
Physics} {\bf 24} (1994), 1325--1346.

\vspace{-2mm}\bibitem{Fuc} L. Fuchs, {\it ``Partially Ordered Algebraic Systems",} Pergamon Press, Oxford-New York, 1963.

\vspace{-2mm}\bibitem{GeIo01}   G. Georgescu and A. Iorgulescu,
Pseudo-MV algebras, {\it Multiple-Valued Logic: An International
Journal} {\bf 6} (2001), 95--135.


\vspace{-2mm}\bibitem{Gla} A.M.W. Glass, {\it ``Partially Ordered Groups",} World Scientific, Singapore, New-Jersey, London, Hong Kong, 1999.

\vspace{-2mm}\bibitem{Good86}  K. R. Goodearl,
{\it ``Partially Ordered Abelian Groups with Interpolation",}
Mathematical Surveys and Monographs No 20, American Mathematical
Society, Providence, Rhode Island, 1986.



\vspace{-2mm}\bibitem{Kalmach83} G. Kalmbach, {\it ``Orthomodular Lattices",}
Academic Press, London, 1983.

\vspace{-2mm}\bibitem{KoCh94}   F. K\^{o}pka and F. Chovanec, D-posets, {\it Mathematica
Slovaca} {\bf 44} (1994), 21--34.

\vspace{-2mm}\bibitem{Hajek03}   P. H\'{a}jek,
Observations on non-commutative fuzzy logic, {\it Soft Computing} {\bf 8}
(2003), 38--43.

\vspace{-2mm}\bibitem{Rac02} J. Rach\r{u}nek, A non-commutative  generalization of
MV-algebras,  {\it Czechoslovak Mathematical Journal} {\bf 52} (2002),
255--273.


\vspace{-2mm}\bibitem{Rie08} Z. Rie\v{c}anov\'{a}, Effect algebraic extensions of generalized
effect algebras and two-valued states,  {\it Fuzzy Sets and
Systems} {\bf159} (2008), 1116--1122.

\vspace{-2mm}\bibitem{RieMar05} Z. Rie\v{c}anov\'{a} and I. Marinov\'{a}, Generalized
homogeneous, prelattice and MV-effect algebras, {\it Kybernetika}
{\bf 41} (2005), 129--142.

\vspace{-2mm}\bibitem{XieLi10a}  Xie Yongjian and  Li Yongming, Riesz ideals in generalized
pseudo-effect algebra and their unitizations, {\it Soft Computing}
{\bf 14} (2010), 387--398.


\vspace{-2mm}\bibitem{XieLi11} Xie Yongjian,  Li Yongming, Guo Jiansheng, Ren Fang and Li Dechao, Weak commutative pseudo-effect algebras,
{\it International Journal of Theoretical Physics} {\bf 50} (2011),
1186--1197.

\vspace{-2mm}\bibitem{li08} Li Yongming, Structures of scale
generalized effect algebras and scale effect algebras, {\it Acta
Mathematica Sinica} {\bf 51} (2008), 863--876. (In Chinese).



\end{thebibliography}
\end{document}